\documentclass[12pt]{article}

\usepackage{amsmath}
\usepackage{amsthm}
\usepackage{amsfonts,amssymb}
\usepackage{epsfig}
\newtheorem{proposition}{Proposition}[section]
\newtheorem{definition}[proposition]{Definition}
\newtheorem{lemma}[proposition]{Lemma}

\newtheorem{corollary}[proposition]{Corollary}

\newdimen\AAdi%
\newbox\AAbo%
\def\AAk#1#2{\setbox\AAbo=\hbox{#2}\AAdi=\wd\AAbo\kern#1\AAdi{}}%
\def\AAr#1#2#3{\setbox\AAbo=\hbox{#2}\AAdi=\ht\AAbo\raise#1\AAdi\hbox{#3}}%
\newcommand {\CA}{{\cal A}}
\newcommand {\CB}{{\cal B}}
\newcommand {\CC}{{\cal C}}

\newcommand {\CF}{{\cal F}}

\newcommand {\CO}{{\cal O}}

\newcommand {\CQ}{{\cal Q}}

\newcommand {\CV}{{\cal V}}

\newcommand{\disp}{\displaystyle}
\newcommand{\eps}{\varepsilon}
\newcommand{\8}{\infty}

\def\lu{\lambda^u}  
\def\ls{\lambda^s}

\def\m1{{-1}}

\newcommand{\ninf}{{n\rightarrow+\8}}

\def\s{\sigma}
\def\l{\lambda}

\newcommand{\wt}{\widetilde}

\newcommand{\ul}[1]{\underline{#1}}

\newcommand{\N}{\Bbb{N}}
\newcommand{\R}{\Bbb{R}}

\newcommand{\Z}{\Bbb{Z}}

\def\picture #1 by #2 (#3){  
\vbox to #2{
\hrule width #1 height 0pt depth 0pt
\vfill
\special{picture #3} 
}}
\def\scaledpicture #1 by #2 (#3 scaled #4){{
\dimen0=#1 \dimen1=#2
\divide\dimen0 by 1000 \multiply\dimen0 by #4
\divide\dimen1 by 1000 \multiply\dimen1 by #4
\picture \dimen0 by \dimen1 (#3 scaled #4)}}

\theoremstyle{definition}
\newtheorem{remark}{Remark}

\usepackage{graphicx,subeqnarray}
\usepackage{epstopdf,psfrag}
\newenvironment{proofof}[1]{\medskip \noindent{\bf Proof of #1.}}{ \hfill\qed\\
}
\newcommand{\bm}{\beta_{max}}
\newcommand{\am}{\alpha_{max}}
\newcommand{\be}{\beta}
\newcommand{\al}{\alpha}

\newcommand{\ie}{{\it i.e.} }

\newcommand{\mr}{\mathfrak{R}}

\def\mbf#1{\mathbf{#1}}
\usepackage{float}

\begin{document}
\synctex=1
\title{Kupka-Smale diffeomorphisms at the boundary of uniform hyperbolicity: a model}
\author{ Renaud Leplaideur and Isabel Lug\~ao Rios}




\maketitle

\begin{abstract}
We construct an explicit example of family of non-uniformly hyperbolic diffeomorphisms, at the boundary of the set of uniformly hyperbolic systems, with  one orbit of cubic heteroclinic tangency.  One of the  leaves involved  in this heteroclinic tangency is periodic, and there is a certain degree of freedom for the choice of the second one. For a non-countable set of choices, this leaf is not periodic and the diffeomorphism is Kupka-Smale:
every periodic point is hyperbolic and the intersections of stable and unstable leaves of periodic points are transverse. 
As a consequence  of our construction, the map is H\"older conjugated to a subshift of finite type, thus every H\"older potential admits a unique associated equilibrium state. 
\end{abstract}

\section{Introduction}


\subsection{Background}

In this work we present an explicit example of a family of diffeomorphisms of the plane having a heteroclinic cubic tangency inside the limit set. Each diffeomorphism of the family is at the boundary of the set of uniformly hyperbolic diffeomorphisms. All the  periodic points are hyperbolic, and the tangency is associated to a periodic leaf and a second leaf that can be chosen periodic or not. 

The construction is part of a project to study dynamic and ergodic properties of bifurcating systems. 
One considers systems at the boundary of the uniformly hyperbolic ones, in such a way that the lack of uniform hyperbolicity appears as a localized phenomenon, and the system satisfies a weaker notion of hyperbolicity (partial hyperbolicity, dominated splitting or non-uniform hyperbolicity).  For these systems one is interested in features as Hausdorff dimension, conformal measures and equilibrium states, which are very well understood for uniformly hyperbolic systems, but still not established for general systems.  The literature is very rich in that domain and we refer the reader to some references and other references therein: \cite{palis-takens, bonatti-diaz-viana, yoccoz-palis, generic} for general and complete surveys of the bifurcation problems, \cite{young-benedicks, alves, alves-bonatti-viana, hu} for the specific problem of SRB measures, \cite{buzzi-thermo, bruin-todd, oliveira, pinheiro, senti-takahasi, varandas} for thermodynamic formalism and Hausdorff dimensions.

In previous works of authors, the mentioned problems were addressed by studying explicit examples or classes of examples. The general case proved not to be  treatable in one single approach, since the answer for some of the questions (for instance, the uniqueness of equilibrium states) is intimately related with the cause of the bifurcation. For this specific problem, there is an indication that the loss of transversality is not as bad as the weakening of expansion and contraction rates, in the sense that many of the results for uniformly hyperbolic systems can be reproduced in the first case, see \cite{isareno1, isareno2}, but not in the second case, see \cite{LOR}.

In \cite{isareno1} and \cite{isareno2}, the authors considered a bifurcating horseshoe displaying a quadratic (internal) homoclinic tangency associated to a fixed hyperbolic saddle. Since the manifolds of the saddle were accumulated by leaves of the set on just one side, it was possible to obtain the tangency as the first bifurcation. Most of the leaves of the horseshoe, though, are accumulated by other leaves on both sides, and a quadratic tangency involving those leaves cannot be reached as the first bifurcation of a one parameter family. In order to deepen our understanding about bifurcating systems, it is important that we are able to study the effects of the presence of a tangency associated to a non periodic leaf.

Systems with cubic heteroclinic tangencies were studied in \cite{Bonatti-Diaz-Vuillemin}, \cite{Enrich} and  \cite{horita-muniz-sabini}. In the two first papers, the authors find cubic tangencies as first bifurcations of Anosov systems. In the third work, the authors prove that the cubic tangencies associated to non-periodic points are somehow prevalent as first bifurcations of Anosov systems, and they also find them in the context of horseshoes. With their results, they prove a conjecture due to Bonatti and Viana (see Conjecture 3.33 in \cite{bonatti-diaz-viana}), that these bifurcations are abundant. In view of their results, it is even more important to understand these systems. 

The main difference between the work of Horita {\it et al.} and the present paper, besides the completely different techniques, is that we have explicit diffeomorphisms, and a complete control of the itinerary of the tangency, which is an important tool for the pursuit of the ergodic properties. 

We point out that there is a Cantor set of choices for one of the leaves involved in the heteroclinic tangency, and for a non-countable number of them, it is non-periodic. The possible choices of the leaves involved in the bifurcation include periodic leaves and leaves with chaotic behaviour. This allows a very rich set of possible dynamic phenomena.

Since the systems studied here are not uniformly hyperbolic, one of the main difficulties is to prove the existence of well defined expanding and contracting directions (for points ouside the critical orbit), and to set the existence and regularity of the invariant manifolds. This is done through the construction of hyperbolic conefields for points that are not in the tangency orbit, and a geometric study of the return maps. This amount of hyperbolicity is enough to prove the existence and uniqueness of equilibrium states for H\"older continuous potentials. As in the case of quadratic tangencies (see \cite{isareno2}), one important potential, the unstable Jacobian, is not in this class. This potential is being considered in a work in progress.

There are two questions we would like to raise. First, are there Kupka-Smale  diffeomorphisms of the plane with quadratic (non periodic) tangencies? Second, through our attempting to find a cubic tangency allowing a good control of hyperbolicity estimates, we ruled out many possibilities. It seems that the low degree terms involving $x$, $y^{3}$ and $yx$ must be the only ones present up to degree 3. Different choices led to increased difficulties to produce the necessary estimates. Is there an intrinsic  reason for that, or is it just an inadequacy of the method?

\subsection{Statement of results}
\label{statements}
Our main result is the following:

\medskip
\noindent
{\bf Theorem}
{\it Any diffeomorphism $F$  of the family $\CA$ to be defined later, satisfies the following properties:
\begin{enumerate}
\item $F$ is at the boundary of the set of uniformly hyperbolic diffeomorphisms, having an orbit of heteroclinic tangency associated to two points in the limit set.
\item It is topologically conjugated to the subshift of finite type with nine symbols and the transition matrix given in \eqref{trans.matrix}.
\item Every point $x$ in the maximal invariant set $\disp\Lambda:=\bigcap_{k\in\Z} F^{-k}([0,1]^{2})$ and out of the orbit of tangency is hyperbolic:
there is an invariant splitting $T_x\R^{2}=E^{u}(x)\oplus E^{s}(x)$, where non-zero vectors in $E^{u}(x)$  expand exponentially,  and vectors in $E^{s}(x)$ contract exponentially.
\item These subbundles are integrable in the sense that every $x\in \Lambda$ admits stable and unstable manifolds, tangent to $E^{s}(x)$ and $E^{u}(x)$, respectively. 
\item The critical orbit is the tangent intersection $W^{u}(P)\cap W^{s}(Q)$. The point $P$ is  a fixed point for $F$ and the point $Q\in\Lambda$ can be chosen in a Cantor set of stable leaves.  
\end{enumerate}
}

The methods here differ from the method in \cite{horita-muniz-sabini}. There, they use parameter exclusion techniques applied to a family of bifurcating systems, to select those with no tangencies between periodic leaves. 
Here, we give the explicit form of the perturbation, describing how the it produces a cubic tangency. The main difficulty in our case is to prove hyperbolicity via the construction of conefields. Our strategy follows the approach in \cite{rios1} for the construction of the map and the conefields, and \cite{isareno1} to get the local stable and unstable manifolds. This explicit construction allows us to choose some properties for the orbit of $Q$. It is, for instance, possible to choose that point to be periodic or non-periodic.  

As a consequence of the exponential rates of expansion and contraction along the invariant manifolds, the conjugacy with the subshift of finite type is a H\"older continuous homeomorphism, then we have the following result.

 \medskip
\noindent
{\bf Corollary}
{\it For every H\"older continuous potential $\phi: \Lambda \to \R$, there exists a unique equilibrium state supported on $\Lambda$.
}

This paper is organized as follows. In Section \ref{sec-def-F} we define the maps $F$ which depend on several parameters. In Section \ref{sec-hyper-F} we construct the stable and unstable conefields and show that the map is (non-uniformly) hyperbolic. In Section \ref{sec-mani} we prove the existence of the stable and unstable manifolds. In Section \ref{sec-border} we prove that the family is in the boundary of non-uniformly hyperbolic diffeomorphisms.

\section{The map and the parameters}\label{sec-def-F}
\subsection{The initial map $F_{0}$}

In this section, we define the maps $F$. First, choose
$0<\lambda<\frac13$, $\s>3$  and $\rho>3$  such that 

\begin{equation}
\label{equ-lambda-sigma}
-1.2<\frac{\log\lambda}{\log\rho}<\frac{\log\lambda}{\log\s}<-1\quad 
\end{equation}

Note that this implies $\lambda\rho<1$ and $\lambda\s<1$. 

Consider a  piecewise linear horseshoe map $F_0$ in $Q=[0,1]^2$ with three components,
as in Figure \ref{fig-partition}. The rectangles in the intersection $F_0(Q)\cap
F_0^{-1}(Q)$ (they will be called {\it rectangles of first generation}) are labeled 
$R_1$ to $R_9$, from the left to the right and from the top to the bottom. They
are assumed to be at distance at least $d$ from each other, and their size (horizontal and vertical) is assumed to be at most $l_0$, with $3l_0+2d\approx 1$.

The horizontal and vertical directions are invariant by $DF_0$. The horizontal
factor of contraction is $\lambda$ for all rectangles, and the vertical
expansion is $\rho$ for $R_1$, $R_2$ and $R_3$, and $\s$ for the other
rectangles of first generation. 

\begin{figure}[htbp]
\begin{center}
\includegraphics[scale=0.5]{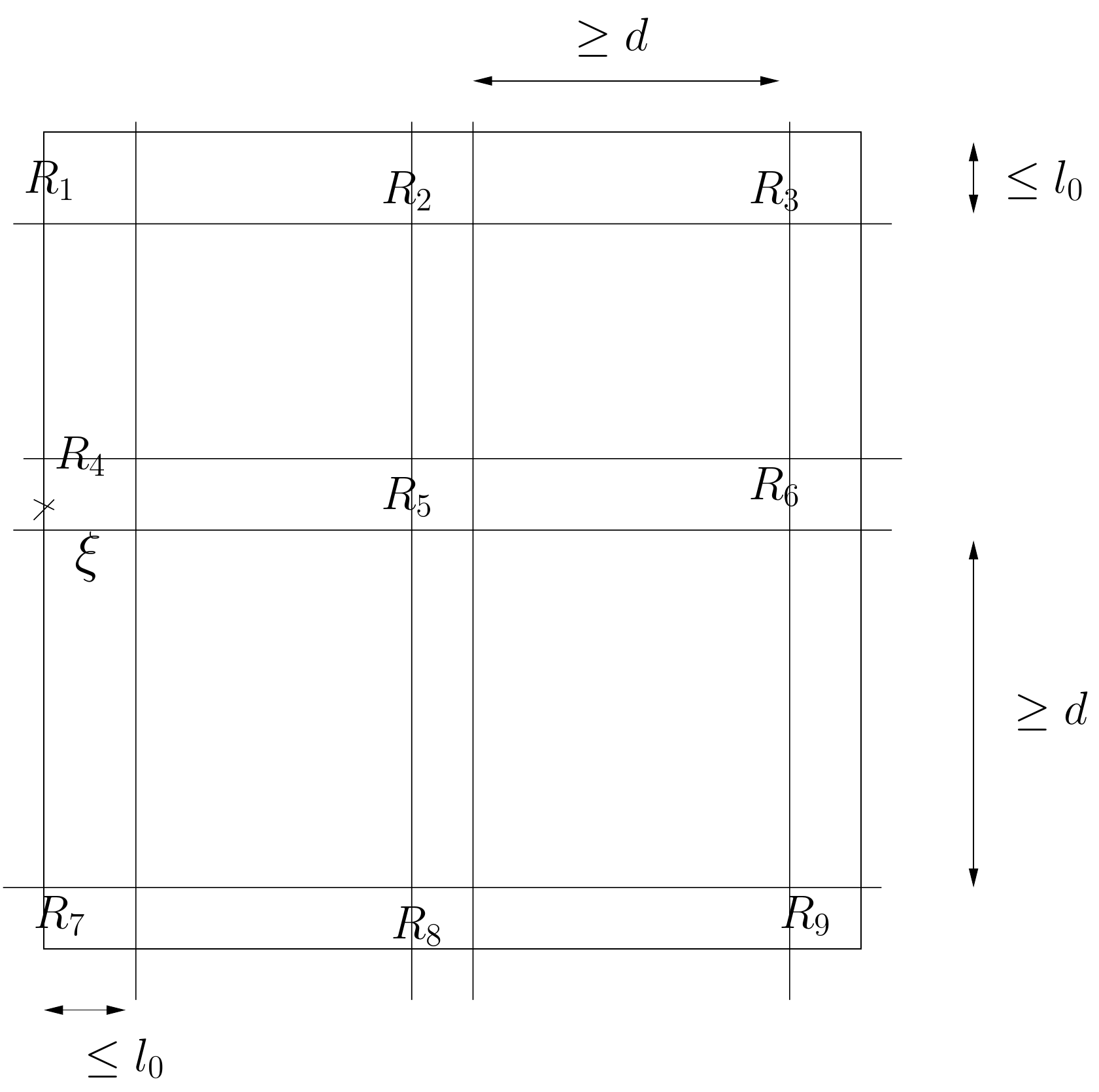}
\includegraphics[scale=0.6]{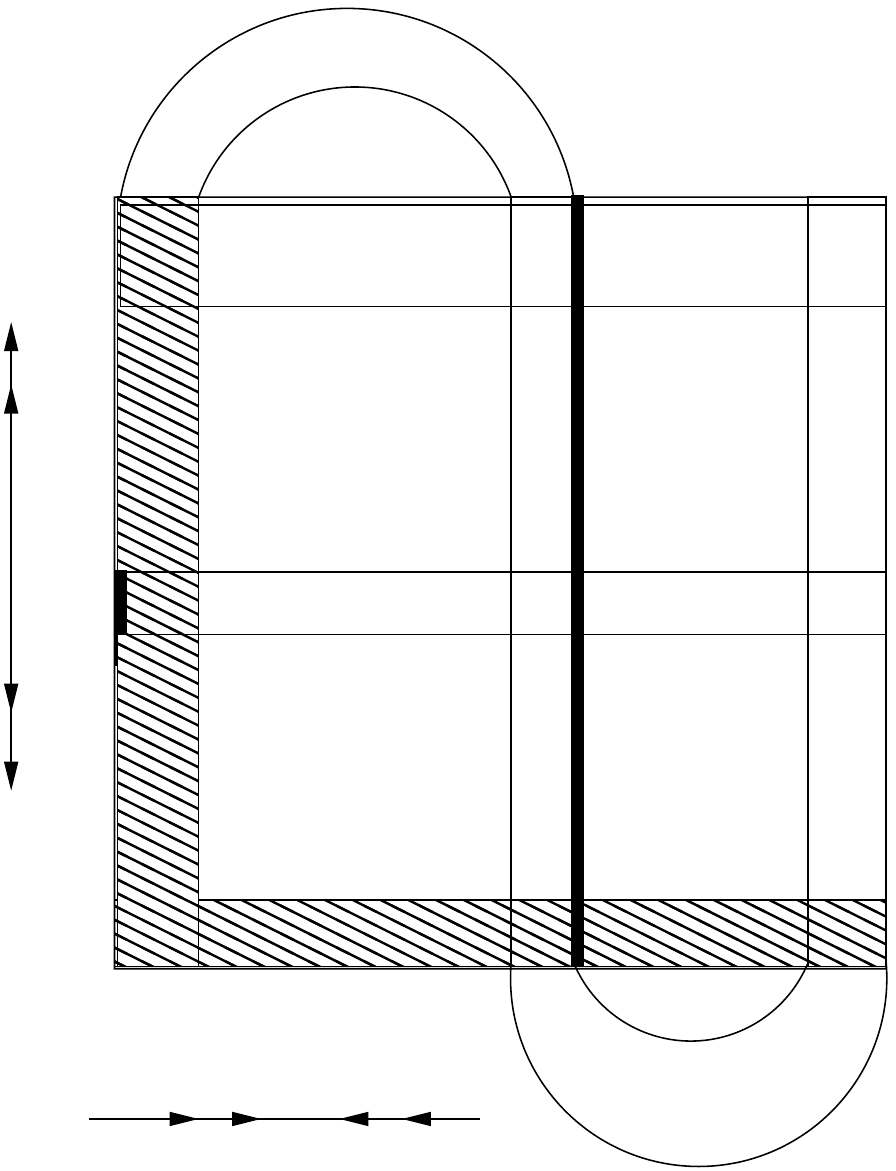}
\caption{The global  initial partition}
\label{fig-partition}
\end{center}
\end{figure}

Now we are going to perturbe $F_0$ in a small region contained
in $R_4$ (to be referred to as the {\it critical region}), to obtain the map $F$.
This small region is sent to $R_8$, and its image will be called {\it
post-critical region}.

 Note that each rectangle $R_i$ is mapped in a vertical stripe that crosses
three rectangles. The transition matrix for this partition is given by 

\begin{equation}\label{trans.matrix}
A=\left( \begin{array}{ccccccccc}0&0&1&0&0&1&0&0&1 \\ 0&0&1&0&0&1&0&0&1 \\ 
0&0&1&0&0&1&0&0&1 \\ 0&1&0&0&1&0&0&1&0 \\ 0&1&0&0&1&0&0&1&0 \\ 0&1&0&0&1&0&0&1&0
\\
 1&0&0&1&0&0&1&0&0 \\ 1&0&0&1&0&0&1&0&0 \\
 1&0&0&1&0&0&1&0&0
\end{array}\right),
\end{equation}
where $A_{i,j}=1$ iff $F_0(R_i)\cap R_j \neq \emptyset$.

This horseshoe-like map is uniformly hyperbolic and conjugated to the subshift
in the space of sequences of nine symbols allowed by the transition matrix. For any $k, n\in\N$, the maximal invariant set in $Q$ is contained in the intersection of the horizontal stripes 
$H_0^k=\{\tau\in Q:F^i_0(\tau)\in Q,\, \forall\,i, 0\leq i\leq k\}$ of generation $k$ with the vertical stripes 
$V_0^n=\{\tau\in Q:F^{-i}_0(\tau)\in Q,\, \forall\,i, 0\leq i\leq n\}$ of generation $n$. Each
connected component of this intersection will be called a {\it rectangle of
generation $(n,k)$}, and the set of all connected components will be denoted by
$G_0^{n,k}$.

Each one of them has horizontal size $\lambda^n$ and the vertical size
between  $\rho^{-k}$ and $\s^{-k}$. More precisely a horizontal stripe $H_{0}^{k}$ has a vertical length
$\s^{-k_{1}}\rho^{-k_{2}}$, where $k_{1}+k_{2}=k$ and are the respective number
of visits of points of the stripe $H_{0}^{k}$ to the less expanding and the most expanding zones (notice that this numbers are the same for all points in $H_{0}^{k}$).

Roughly speaking, point in $H^{k}_{0}$ have the same itinerary for the next forward $k$ iterates of $F_{0}$ and points in $V^{n}_{0}$ have the same itinerary for the next $n$ backward iterates. Points in some rectangle of $G_0^{n,k}$ have the same itinerary from $-n$ to $+k$.

The image by $F_0$ of one rectangle in
$G_0^{n,k}$ intersects exactly three rectangles of the same generation, and
contains (determines) three rectangles of generation $(n+1,k)$.

\medskip

We recall that $d$ is a lower bound for the horizontal and vertical gaps  between the elements of the partition $G_{0}^{1,1}$ ; $l_{0}$ is an upper bound for the horizontal and vertical length of the elements of $G_{0}^{1,1}$. 
We emphasize that $d$ can be chosen as
close as wanted to $\frac12^{-}$ by doing $\s$, $\rho$ and $\frac1\lambda$ 
big enough. Similarly $l_{0}$ can be chosen as small as needed. We shall use this later. 


\subsection{Definition of the map $F$}

In this subsection we define the map $F$, depending on the parameters $\l$,
$\s$, $\rho$ and some others to be introduced. We are going to assume some conditions on the parameters to be stated along the way, always respecting the
conditions stated untill now, including Equation (\ref{equ-lambda-sigma}).

Note that each rectangle $R_i$ is a union of horizontal lines, each one consisting of
points with the same future by $F_0$ (some of them are sent outside $Q$ by some
future iterate).

\subsubsection{The critical region}

We choose a point $\xi '$ belonging to the right hand side of $R_8$. We choose it such
that 
\begin{enumerate}
\item its future itinerary never intersects $R_7\cup R_4$,
\item its future itinerary intersects infinitely many times the set $R_5\cup R_6$,
\item its future itinerary intersects infinitely many times the set $R_1\cup R_2\cup R_3$.
\end{enumerate}
Its preimage $\xi=(0,\xi_2)$  belongs to the
left hand side of $R_4$, and all backward iterates of $\xi$ are in the left hand
side of $R_7$, converging to the fixed point $(0,0)$.

\paragraph{Notation.} We will denote by $\CO(\xi)$ the critical orbit, \ie 
$$\CO(\xi):=\left\{F^{k}(\xi),\ k\in\Z\right\}.$$
We set $\CO^{\pm}(\xi)$ for the forward or backward orbit of $\xi$.

\bigskip
We change the map $F_0$ in a rectangle
$\mathfrak{R}=[0,\alpha_{max}]\times[\xi_2-\beta_{max}, \xi_2+\beta_{max}]$ (see  Figure \ref{fig-pertudyna} ).
This region is referred to as the \emph{critical region}, and $\xi$ is the \emph{critical point}. 
The constants $\alpha_{max}$ and
$\beta_{max}$ must satisfy several conditions, to be stated along the way.

\begin{figure}[htbp]
\begin{center}
\includegraphics[scale=0.45]{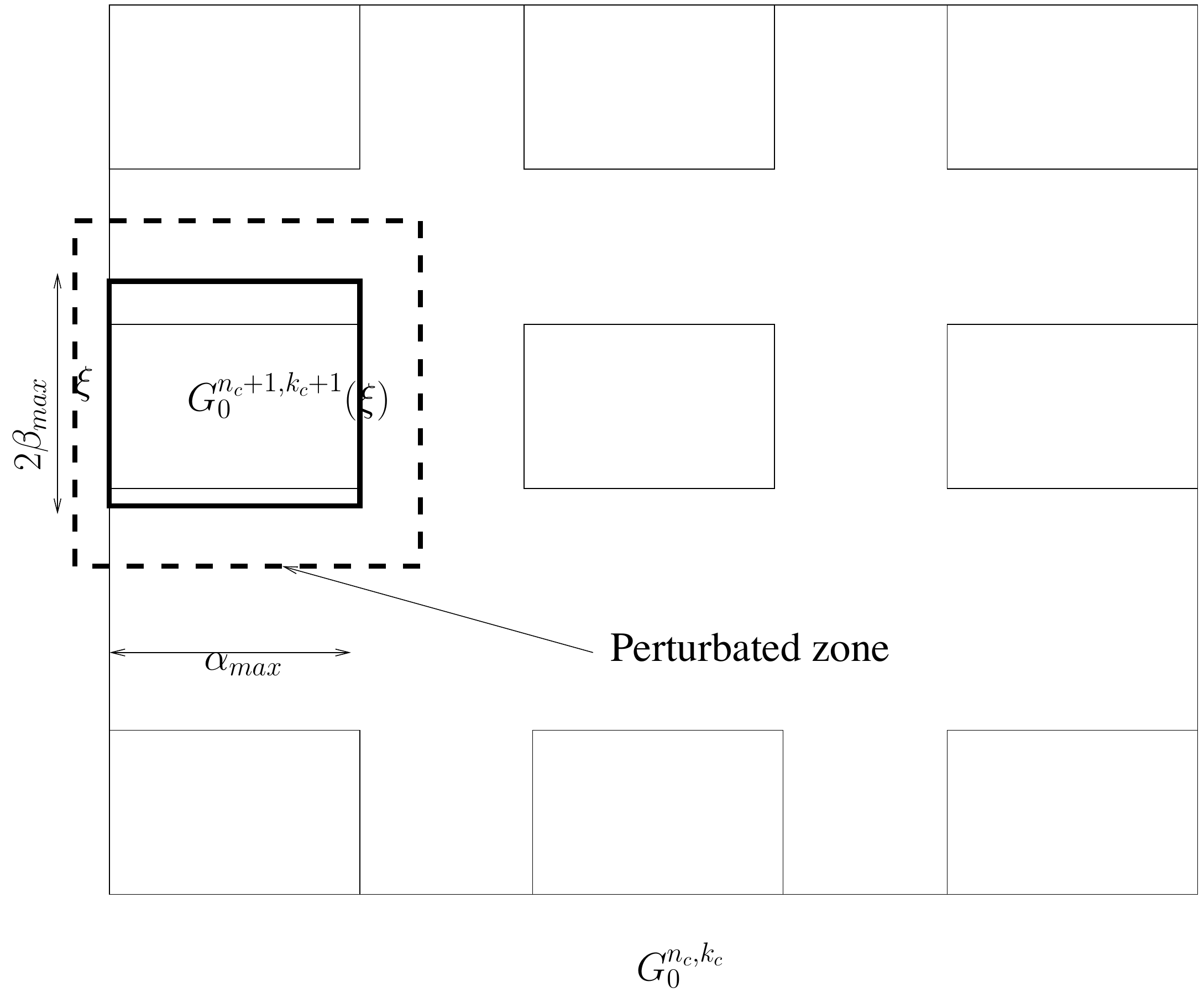}


\caption{Choice of the critical region}
\label{fig-pertudyna}
\end{center}
\end{figure}

\bigskip
Related to these quantities we introduce two integers $n_{c}$ and $k_{c}$. 
We ask $\mr$ to contain exactly one rectangle of $G_0^{n_{c}+1,k_{c}+1}$, 
 namely $G_{0}^{n_c+1,k_c+1}(\xi)$, and to have empty intersection with all the other rectangles of
$G_0^{n_{c}+1,k_{c}+1}$. 

It also has to be contained into a rectangle $G_{0}^{n_c,k_c}$, see Figure \ref{fig-pertudyna}. For this, we choose a value of $k$ such that $F_0^{k+1}(\xi)\in R_5\cup R_6$.
Doing this, we can be sure that $G_{0}^{n_c+1,k_c+1}(\xi)$ is contained in the
interior of the rectangle $G_{0}^{n_c,k_c}(\xi)$, except by the left boundary of
$G_{0}^{n_c+1,k_c+1}(\xi)$, that is contained in the interior of the left boundary
of $G_{0}^{n_c,k_c}(\xi)$. Other conditions on $n_{c}$ and $k_{c}$ will be also stated
later. 

We consider $k_{1}$ and $k_{2}$ such that $k_{1}+1+k_{2}=k_{c}+1$  corresponding to
the number of visits into the less and more expanding zones of
$G_{0}^{n_c+1,k_c+1}$ by the forward $k_c+1$ iterations by $F_0$. By construction, the first expansion is $\s$. The alternation of big ($\rho$) or small ($\s$) expansions are similar for every point in $G_{0}^{n_c,k_c}(\xi)$ for the first $k_c$ iterates of $F_0$.

The condition $\disp G_{0}^{n_c+1,k_c+1}(\xi)\subset \mr\subset G_{0}^{n_c,k_c}(\xi)$ and $\disp G_{0}^{n_c+1,k_c+1}(M)\cap \mr=\emptyset$ if  $G_{0}^{n_c+1,k_c+1}(\xi)\neq G_{0}^{n_c+1,k_c+1}(M)$ can be realized if   
\begin{subeqnarray}\label{condisuff-1}
 l_{0}\s^{-k_{1}}\rho^{-k_{2}}&\le& \beta_{max}\le
d\s^{-k_{1}}\rho^{-k_{2}},\slabel{equ1-betamax-genen+1k+1}\\
 \alpha_{\max}&=& l_{0}\lambda^{n_{c}}\slabel{equ1-almax-genen+1k+1}.
\end{subeqnarray}
 These two conditions are referred to as
conditions \eqref{condisuff-1}.

\bigskip
The image of $G_{0}^{n_c+1,k_c+1}(\xi)$ by $F_{0}$ is a rectangle of
$G_{0}^{n_c+2,k_c}$. The image of $G_{0}^{n_{c}+1,k_{c}}(\xi)$ is the union of three rectangles of $G_{0}^{n_c+2,k_c}$: the one which contains $F_0(\xi)$ and the ones immediately above and  below.

\subsubsection{Dynamics in the critical region}

The map $F$ is equal to $F_{0}$ outside a neighborhood of $\mr$. In $\mr$ it is defined  as
follows (see Figure \ref{fig-pertudyna-2}). 
The  vertical lines in $\mr$ are sent over cubic curves and the horizontal lines
are sent over verticals. This is done in such a way that $F(\xi)$ is
in the same horizontal line than $F _{0}(\xi)=\xi'$. As $\xi'$ never comes back to $R_{4}$ by iteration of $F_{0}$, for every $n\ge 0$, $F^{n}(F(\xi))=F_{0}^{n}(F(\xi))$ and then 
$F^{n+1}(\xi)$ and $F^{n}_{0}(\xi ')$ always belong to the same $R_{j}$. We ask for $F(\xi)$ to be centered with respect
to the vertical stripe $V_{0}^{n_{c}+2}(\xi ')$, in the horizontal direction (again, see Figure \ref{fig-pertudyna-2}).  For 
$\xi+(x,y)\in \mr$,
we set 

$$F(\xi+(x,y))=F(\xi)+(-\eps_{1} y,bx-cy(y^2+x)).$$
 The constants $c>1$ and $\eps_{1}<1$  and $b$ are positive parameters of the
map. For simplicity we set $\disp \CF(x,y)=(-\eps_{1} y,bx-cy(y^2+x))$. In other words,
we have
 $$F(\xi+(x,y))=F(\xi)+\CF(x,y).$$
 
 \begin{figure}[htbp]
\begin{center}

\includegraphics[scale=0.5]{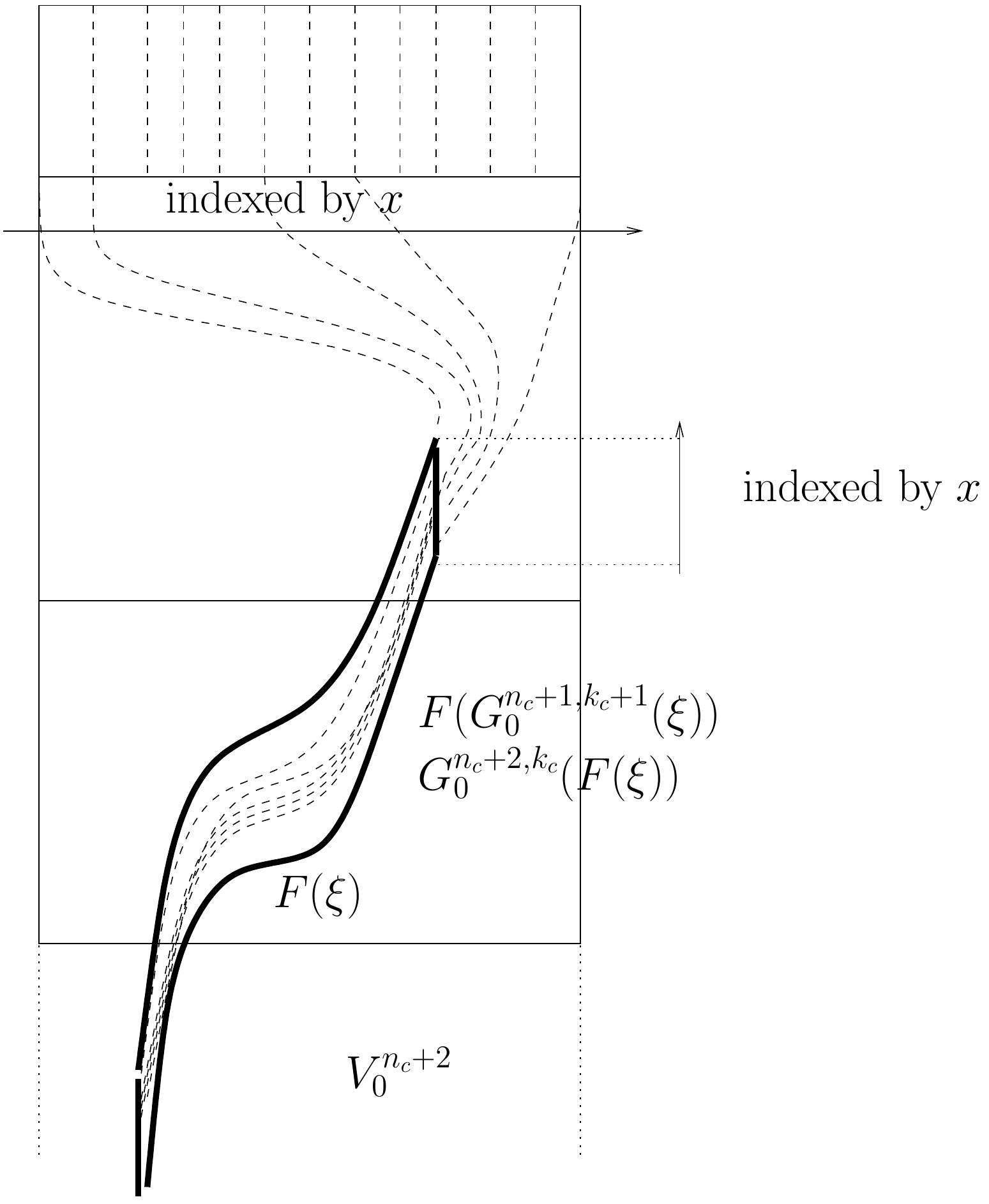}

\caption{Perturbation of the dynamics}
\label{fig-pertudyna-2}
\end{center}
\end{figure}

 We introduce a new parameter $A$ which is a positive real number. It will be
used to define the unstable cone field. We assume that the parameters satisfy
the following conditions:
\begin{subeqnarray}\label{condi-suff}
1<&3c\beta_{max}<&\frac1{\eps_{1}},\slabel{equ-cbeta-lambda}\\
b&=&2c\beta_{max},\slabel{condi-b-c-betamax}\\
1&<<&\frac{Ad\eps_{1}}{3c\beta_{max}},\slabel{equ-grandApour expansion}\\
1<<A&=&\frac{c}{8\eps_{1}},\slabel{condi-majo Abetamax}
\end{subeqnarray}
This conditions will be referred to as system of conditions \eqref{condi-suff}.

\bigskip 
To control the future dynamics of $F$ using the dynamics of $F_0$ we ask for that the image by $F$ of $\mr$ stays into the interior of the vertical band $V_0^{n_{c}+2}(F_0(\xi))$ (see Figure \ref{fig-pertudyna-2}). We also ask for that $F(\mr)$ only intersects one element of $G_0^{n_{c}+2,k_{c}}$ and that this intersection occurs in a special way (describe just below).

All the cubic curves we consider are graphs over the horizontal interval of length
$2\eps_{1}\beta_{\max}$ centered at $F(\xi)$. On the left hand side these curves finish in a vertical segment,
which is the image of the horizontal segment
$[0,\alpha_{max}]\times\{\xi_{2}+\beta_{max}\}$. Its points are of the form $ (-\eps_1\beta_{max},bx-c\bm(\bm^2+x))$, where $x\in[0,\am]$. 
Similarly, on the right hand side there
is the vertical segment image of the horizontal segment
$[0,\alpha_{max}]\times\{\xi_{2}-\beta_{max}\}$. Its points are of the form $ (\eps_1\beta_{max},bx+c\bm(\bm^2+x))$, where $x\in[0,\am]$. 

These two segments are not at the same vertical position. If $\alpha_{max}$ is
chosen very small, equality \eqref{condi-b-c-betamax} shows that the segment from the right has its
bottom higher than the top of the segment from the left. We want
that the horizontal line $H_{0}^{k_{c}}(F_{0}(\xi))$ passes through these two
vertical segments. 

These conditions are realized if the following system holds:
\begin{subeqnarray}\label{condisuff-2}
-c\bm^3>-\frac{2d}3\s^{-k_1+1}\rho^{-k_2}&\hskip-.3cm,&\hskip-.5cm c\bm^3+3c\bm\am<\frac{2d}3\s^{-k_1+1}\rho^{-k_2}
\slabel{equ2-betamax-genen+1k}, \\
2\eps_{1}\beta_{max}&<&
l_{0}\lambda^{n_{c}+1}\slabel{equ3-betamax-genen+1k},\\
2l_{0}\s^{-k_{1}+1}\rho^{-k_{2}}+c\beta_{max}\alpha_{max}-c\bm^3&<&c\beta_{max}^{3}\slabel{equ4-betamax-genen+1k}.
\end{subeqnarray}
These conditions are referred to as conditions \eqref{condisuff-2}.
Condition \eqref{equ4-betamax-genen+1k} is the exact one to ensure that the two
vertical segments are torn apart by the horizontal stripe
$H_{0}^{k_{c}}(F_{0}(\xi))$. Condition \eqref{equ3-betamax-genen+1k} means that the image $F(\mr)$ can be included into the vertical band $V_0^{n_{c}+2}(F_0(\xi))$. Condition \eqref{equ2-betamax-genen+1k} ensures that the image $F(\mr)$ does not intersect other rectangles of $G_0^{n_{c}+2,k_{c}}$. 

Moreover, as the conditions hold with strict inequalities, the perturbation $F$
 of $F_0$ (for the $\CC^0$-topology) can be made as regular as wanted.

\subsubsection{Dynamics outside the critical region}

To finish to describe the map in $R_{4}$ we have to explain how it is defined
inside the gaps between the elements of $G_0^{n_{c}+1,k_{c}+1}$ of $G_0^{n_{c},k_{c}}(\xi)$. 
Note that the perturbation is only done on a small neighborhood of $G_0^{n_{c}+1,k_{c}+1}(\xi)$ which does not intersect other elements of $G_0^{n_{c}+1,k_{c}+1}$. 

This neighborhood is however close to 3 other rectangles of $G_0^{n_{c}+1,k_{c}+1}$, the one at the right hand side of $G_0^{n_{c}+1,k_{c}+1}(\xi)$, the one directly above and the one directly below.

These two last rectangles are laminated by vertical lines, indexed by $x$. Their image by $F=F_0$ are also laminated by vertical lines, still indexed by $x$ (up to a linear scaling). 
In $G_0^{n_{c}+1,k_{c}+1}(\xi)$, the corresponding vertical lines are sent on cubics. They get ``out'' $F(\xi+[0,\am]\times[-\bm,\bm])$ by the two vertical segments describe above; this is still indexed by $x$. 
Then, we simply connect the line of index $x$ to the cubic with the same index, and do it such that the global envelop stays into the vertical band $V_0^{n_{c}+2}$ (see Figure \ref{fig-pertudyna-2}). 
This ensures that points which are not in the horizontal strip $H_0^{k_{c}}(\xi)$  have the same image by $F$ or $F_0$.

\begin{figure}[htbp]
\begin{center}
\includegraphics[scale=0.5]{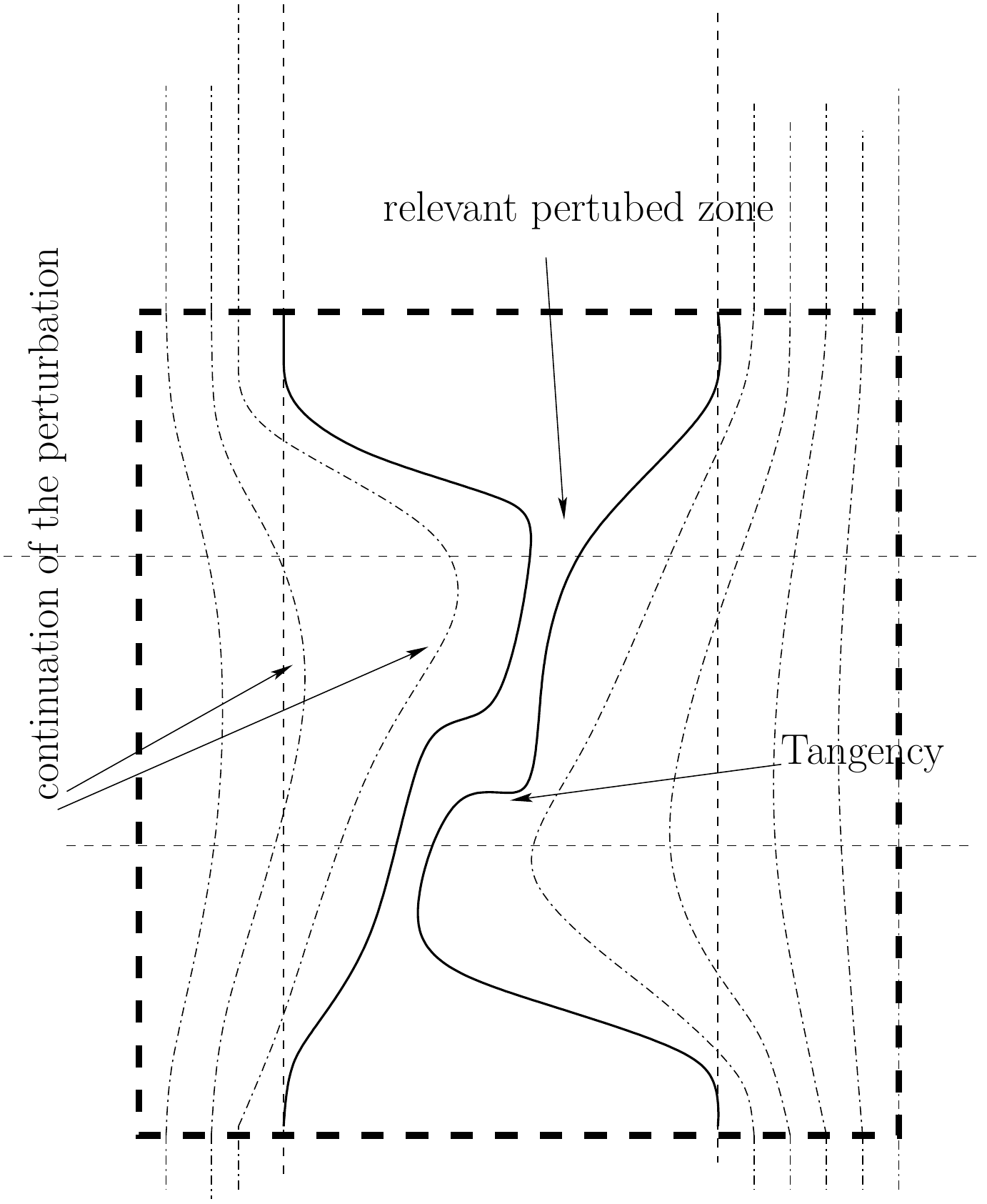}
\caption{Image of the perturbation of the dynamics outside $G_0^{n_{c}+1,k_{c}+1}(\xi)$ and inside the perturbed zone in Fig. \ref{fig-pertudyna}.}
\label{fig-pertu-image}
\end{center}
\end{figure}

We still use this vertical lamination to explain the perturbation of the dynamics on both lateral sides of $G_0^{n_{c}+1,k_{c}+1}(\xi)$. The vertical lines are sent on curves as shown on Figure \ref{fig-pertu-image}.

The main condition we ask for to occur is that points which are between the gaps of rectangles of $G_0^{n_{c}+2,k_{c}}$ do not belong to the image $F(\CQ)$.

\subsection{Set of parameters} 
We can now prove that the set of parameters satisfying our conditions is large. 
\begin{lemma}
\label{lem-condi-openset}
The four series of conditions \eqref{equ-lambda-sigma},\eqref{condisuff-1},
\eqref{condi-suff} and \eqref{condisuff-2} are available for an open set of
parameters. This set of parameters allows $c$, $\s$ and $\rho$ to increase,
$\eps_{1}$ and $\lambda$ to decrease. 
\end{lemma}
\begin{proof}
We set 
\begin{equation}
\label{equ-def-betamax-gene}
\beta_{max}=\frac{d}2\s^{-k_{1}+1}\rho^{-k_{2}}.
\end{equation}
We emphasize that this equality means that $k_{c}$ is fixed.

We assume that $\s$, $\rho$ and $\frac1\lambda$ are very big. Remember that
$l_{0}\approx\lambda$ and $d\approx \frac12$. This means that
\eqref{equ1-betamax-genen+1k+1} holds (and one can increase $\s$, $\rho$ and
$\frac1\lambda$).
Moreover, our condition on the existence of infinitely many forward iterates  for $F_0$ such that $F_0^{k_{c}}(\xi)$ belongs to $R_5$ or $R_6$ shows that $k_{c}$ can be taken as big as wanted. Hence, $\bm$ is as small as wanted. Note that condition \eqref{condi-b-c-betamax} fixes
the value of $b$ as soon as $c\beta_{max}$ is chosen.

Due to  \eqref{equ1-almax-genen+1k+1}, there exists a unique integer $n_{c}$ such that 
\begin{equation}
\label{equ-def-alphamax-fonc-betamax}
\frac{\lambda.\be_{max}^{2}}{10}<\al_{max}\le \frac{\be_{max}^{2}}{10}.
\end{equation}
This fixes $\al_{max}$ as soon as $\be_{max}$ is fixed. We assume that $\bm$ is chosen sufficiently small such that $n_c$ is bigger than 5.

Our choice for $\am$ yields that if the following system of inequalities  holds,

\begin{subeqnarray}\label{condisuff-2.2}
c\bm^3&<&\frac{2d}3\s^{-k_1+1}\rho^{-k_2}
\slabel{equ2-2-betamax-genen+1k}, \\
2l_{0}\s^{-k_{1}}\rho^{-k_{2}}&<&2c\beta_{max}^{3}\slabel{equ4-2-betamax-genen+1k},
\end{subeqnarray}
then inequalities \eqref{equ2-betamax-genen+1k} and \eqref{equ4-betamax-genen+1k} also hold.

%

Now, we set 
\begin{equation}
\label{equ-def-cbeta2}
1=c\beta_{max}^{2}.
\end{equation}
The  definition of $\be_{max}$ in \eqref{equ-def-betamax-gene} and  $d\approx \frac12$ and $l_0\s\approx 1$ show that  \eqref{equ4-2-betamax-genen+1k} holds (for $\s$ sufficiently big).

As $\beta_{max}$ is very small, $c\beta_{max}>1$ and part of
\eqref{equ-cbeta-lambda} holds. 
Note that $\disp \frac{Ad\eps_1}{3c\bm}=\frac{d}{24\bm}$ and is as big as wanted if $\bm$ decreases.

Now, to satisfy \eqref{equ3-betamax-genen+1k}, \eqref{equ-cbeta-lambda} and \eqref{condi-majo Abetamax}, it is sufficient that the following holds (remember \eqref{equ1-almax-genen+1k+1}):
\begin{eqnarray*}
 2\eps_1\bm&<&\lambda\am,\\
15\eps_1&<&\bm,\\
\eps_1\bm^2&<&\frac18.
\end{eqnarray*}
All these conditions hold if, for a given $\bm$, $\eps_1$ is chosen sufficiently small.

 To conclude the proof, note that $c$ can increase as wanted (if $\beta_{max}$
decreases), $\s$ can  also increase (then $\beta_{max}$  and $\lambda$ decrease
and $\rho$ increases), $\eps_{1}$ can decrease. Integers $k_{c}$ and $n_{c}$ can be chosen as big as wanted, up to the restriction on $k_{c}$ saying that $F_0^{k_{c}+1}(\xi)\in R_5\cup R_6$. 

\end{proof}

\begin{remark}
\label{rem-choice parameters}
Later we will ask for more conditions on the parameters. In particular we shall
ask that $A$ is big enough, with respect to other parameters. As we pointed out, this can always be done by increasing the value of $c$, letting
$c\beta_{max}^2$ constant or decreasing $\eps_1$
$\blacksquare$\end{remark}

\section{Cone fields, expansion and contraction}\label{sec-hyper-F}

In this section we construct the unstable and the stable conefields and prove expansion and contraction in the hyperbolic splitting. The construction is stated for  points in $\CQ\cap F(\CQ)\cap F^{-1}(\CQ)$ and outside the critical orbit. The unstable conefield is invariant by $DF$, in the sense that the image of the cone at $M$ by $DF(M)$ is included in the cone at $F(M)$. Analogously, the stable conefield is invariant by $DF^{-1}$. The properties of expansion and contraction of vectors inside the unstable and stable cones will be proved only along pieces of orbits that stay inside $\CQ$ long enough. As a consequence, there is an invariant splitting into stable and unstable subspaces of the tangent space at points of $\Lambda$ out of the critical orbit.

\subsection{Dynamical partition for $F$}

Now consider the sets
 
$$H^{k}=\{\tau\in \CQ:F^i(\tau)\in \CQ,\, \forall\,i, 0\leq i\leq
k\},$$ 
$$V^{n}=\{\tau\in \CQ:F^{-i}(\tau)\in \CQ,\, \forall\,i, 0\leq i\leq n\}$$ 
and 
$G^{n_{c},k_{c}}$ the set of connected components of $H^k\cap V^n$. Notice that they are defined using the perturbed map $F$.

\begin{lemma}\label{lem-comp-bands}
 Let $k_c$ and $n_c$ be as in the choice of the region $\mr$. Then $\CQ\setminus H^{k_{c}}$ contains $\CQ\setminus H_0^{k_{c}}$. Similarly $\CQ\setminus V^{n_{c}+2}$ contains $\CQ\setminus V_0^{n_{c}+2}$.
\end{lemma}
\begin{proof}
 Let us consider a point $M$ in $\CQ\setminus H_0^{k_{c}}$. By definition, there exists $j<k_{c}$ such that $F_0^{j-1}(M)$ belongs to $\CQ$ but $F_0^j(M)\notin \CQ$. 

This first shows that $M$ does not belong to $H_0^{k_{c}}(\xi)$, otherwise we could iterate $F_0$ for at least $k_{c}$-times. In particular $F(M)=F_0(M)$, and $F_0(M)$ belongs to $\CQ\setminus H_0^{k_{c}-1}$. Again, this shows that $F(M)=F_0(M)$ does not belong to  $H_0^{k_{c}}(\xi)$, and $F^2(M)=F_{0}^2(M)$. By induction, we finally get $F^j(M)=F_0^j(M)\notin \CQ$. 

The proof for $\CQ\setminus V^{n_{c}+2}$ is the same, up to the fact that we iterate $F^{-1}$ and we have to consider the vertical band $V_0^{n_{c}+2}(F_0(\xi))$. 

\end{proof}

Now consider the set $\Lambda:=\bigcap_{n\in\Z}F^n([0,1]^2)$.

\begin{corollary}\label{cor-tempo-volta}
A point $M\in \mr\cap \Lambda$ has a first return time in $\mr$ greater than  or equal to $k_c+1$.
\end{corollary}
\begin{proof}
Pick a point $M$ in $\mr$. If its image by $F$ is outside $G_0^{n_c+2,k_c}(F(\xi))$, then, it belongs to  $\CQ\setminus H_0^{k_c}$ and  by Lemma \ref{lem-comp-bands} it does not belong to $\Lambda$.

If it belongs to $G_0^{n_c+2,k_c}(F(\xi))$, then $F_0^j(F(M))$ belongs to the same $R_i$ than $F_0^j(\xi')$ for $j\le k_c-1$. Therefore $F_0^j(F(M))$ does not come back to $R_4$ for $j\le k_c-1$ and $R_4$ is the unique place where $F$ and $F_0$ are different. Thus $F^{j+1}(M)$ coincides with $F_0^j(F(M))$ for $j\le k_c-1$, and the first return time in $R_4$ is at least $k_c+1$. 
\end{proof}

\begin{remark}
 \label{rem-rectangle=gg0}
An immediate corollary of Lemma \ref{lem-comp-bands} is that rectangles of $G^{n_c+2,k_c}$ are inside rectangles of $G_0^{n_c+2,k_c}$. 
$\blacksquare$\end{remark}

\subsection{Unstable conefield at the post-critical region}

Consider $x$ and $y$ the local coordinates close to $\xi$. We define the unstable conefield close to $F(\xi)$ by 
assigning for $F(\xi+(x,y))$ the cone
 
$$\mathsf{C}^{u}(F(\xi)+\CF(x,y))=\{\mathbf{v }=(v_{x},v_{y}),\ |v_{y}|\ge
A(3y^2+x)|v_{x}|\},$$
where $A$ is the parameter introduced above. This cone is centered on the vertical direction, and contains the tangent direction at $F(\xi+(x,y))$ to the cubic curve $F(\{(x,t), t\in [0,1]\}$. It is not defined for $F(\xi)$.

\begin{lemma}
\label{lem-cone-instable}
Assume that the parameters satisfy conditions \eqref{condi-suff}. Let $\xi+(x,y)$ be
in the critical zone and returning for the first time in it in
$\xi+(\alpha,\beta)=:F^{k_{0}+1}(\xi+(x,y))$; then,
$DF^{k_{0}+1}(\mathsf{C}^{u}(F(\xi+(x,y))))\subset \mathsf{C}^{u}(F(\xi+(\alpha,\beta)))$.
\end{lemma}
\begin{proof}
Let us pick $M=F(\xi)+\CF(x,y)$, and assume that the first entrance of $M$ in the critical zone is on the point
$F^{k_{0}}(M):=\xi+(\alpha,\beta)$. We pick a vector in $\mathsf{C}^{u}(M)$, of the form 
$$\mathbf{v}=\left(\begin{array}{c}1 \\Au(3y^2+x)\end{array}\right),$$
with $|u|\ge 1$ (note that any vector of  $\mathsf{C}^{u}(M)$ is proportional to one of this form). 

First we note that $\alpha\neq 0$, otherwise for every $n\ge 1$,
$F^{-n}(\xi+(\alpha,\beta))$ would belong to the rectangle $R_{7}$. This produces
a contradiction with the fact that $F^{-k_{0}}(\xi+(\alpha,\beta))=M$.

We compute the derivatives along the piece of orbit from $M$ to
$F^{k_{0}+1}(M)$.
For $0\le k\le k_{0}-1$, $F^{k}(M)$ belongs to the region where $F$ is linear. Then
\begin{equation}
\label{equ1-cone-instable}
DF^{k_{0}}(M).\mbf{v}:= \left(\begin{array}{c}\lambda^{k_{0}}
\\A\s^{k_{0}^{(1)}}\rho^{k_{0}^{(2)}}u(3y^2+x)\end{array}\right),
\end{equation}
where ${k_{0}^{(1)}}$ and ${k_{0}^{(2)}}$ are respectively the number of visits to the less expanding and more expanding rectangles. The derivative at $\xi+(\alpha,\beta)=F^{k_{0}}(M)$ is
 \begin{equation}
\label{equ2-cone-instable}
\left(\begin{array}{cc}0 &  -\eps_{1}\\b-c\beta &
-c(3\beta^2+\alpha)\end{array}\right).
\end{equation}
We thus get
\begin{eqnarray}
DF^{k_{0}+1}(M).\mbf{v}&=& \left(\begin{array}{cc}
0 & -\eps_{1} \\b-c\beta & -c(3\beta^2+\alpha)
\end{array}\right).
\left(\begin{array}{c}
\lambda^{k_{0}} \\
A\s^{k_{0}^{(1)}}\rho^{k_{0}^{(2)}}u(3y^2+x)
\end{array}\right)\nonumber\\
&=& \left(\begin{array}{c}-\eps_{1}.A \s^{k_{0}^{(1)}}\rho^{k_{0}^{(2)}} u(3y^2+x)
\\\lambda^{k_{0}}(b-c\beta)-c(3\beta^2+\alpha)A\s^{k_{0}^{(1)}}\rho^{k_{0}^{(2)}}u(3y^2+x)\end{array}
\right).
\label{equ3-cone-instable}
\end{eqnarray}
Our goal is to prove that the right hand side member of
\eqref{equ3-cone-instable} belongs to $\mathsf{C}^{u}(F(\xi+(\alpha,\beta)))$. It is
thus sufficient to prove that with our parameters, inequality
\begin{equation}
\label{equ4-cone-instable}
\left|\frac{\lambda^{k_{0}}(b-c\beta)}{\s^{k_{0}^{(1)}}\rho^{k_{0}^{(2)}}Au(3y^2+x)\eps_{1}}-\frac{c}{
\eps_{1}}(3\beta^2+\alpha)\right|\ge A(3\beta^2+\alpha)
\end{equation}
holds. The main idea is to prove that if the parameters satisfy conditions
\eqref{condi-suff} then,  the  first term in the left hand side of \eqref{equ4-cone-instable} is
smaller than  half of the second term in the left hand side. Namely we want to prove that our
assumptions yield 
\begin{equation}
\label{equ-majotermesmoitie}
0\le \frac{\lambda^{k_{0}}(b-c\beta)}{\s^{k_{0}^{(1)}}\rho^{k_{0}^{(2)}}Au(3y^2+x)\eps_{1}}\le
\frac{c}{2\eps_{1}}(3\beta^2+\alpha).
\end{equation}

Indeed, assuming that \eqref{equ-majotermesmoitie} holds, \eqref{condi-majo Abetamax}
shows that \eqref{equ4-cone-instable} also holds because the left hand side term is at least $\disp\frac{c}{2\eps_{1}}(3\be^{2}+\al)$ and the right hand side is $\disp\frac{c}{8\eps_{1}}(3\be^{2}+\al)$: The left hand size term in at least $\disp \disp\frac{c}{2\eps_{1}}(3\be^{2}+\al)$, whereas the left hand size term is $\disp\frac{c}{8\eps_{1}}(3\be^{2}+\al)$.

\bigskip
Let us now prove that \eqref{equ-majotermesmoitie} holds.
The fact that $k_{0}$ is the first visit for $M$ close to the critical point (\ie in $\mr$)
yields some additional inequalities.
First, we claim that the vertical distance between $F^{k_{0}}(M)$ and
$F^{k_{0}+1}(\xi)$ is strictly larger than $d$, otherwise these two
points would belong to the same $R_{i}$. This is impossible
because the forward orbit of the critical point never returns to the element
$R_{4}$.
This yields:
\begin{equation}
\label{equ5-cone-instable}
\s^{k_{0}^{(1)}}\rho^{k_{0}^{(2)}}|bx-cy(y^{2}+x)|>d.
\end{equation}

Similarly, the backward orbit of the critical point never returns to the element
$R_{4}$, hence the horizontal distance between $F^{-k_{0}}(\xi)$ and
$F^{-k_{0}}(\xi+(\alpha,\beta))$ is larger than $d$. This yields:
\begin{equation}
\label{equ6-cone-instable}
\frac\alpha{\lambda^{k_{0}}}>d.
\end{equation}
Note that \eqref{equ6-cone-instable} is equivalent to
$\lambda^{k_{0}}<\disp\frac1{d}\alpha$. Reporting this in $\disp 
\frac{\lambda^{k_{0}}(b-c\beta)}{\s^{k_{0}^{(1)}}\rho^{k_{0}^{(2)}}Au(3y^2+x)\eps_{1}}$, our goal is
reached if 
$\disp| \s^{k_{0}^{(1)}}\rho^{k_{0}^{(2)}}Au(3y^2+x)|$ is uniformly large (we obviously have $\al\le 3\be^2+\al$)

\paragraph{The case $y\le 0$.}
As $x$ is non-negative, inequality \eqref{equ5-cone-instable} is equivalent to
$$\s^{k_{0}^{(1)}}\rho^{k_{0}^{(2)}}(bx+c|y|(y^2+x))>d.$$
This and \eqref{condi-b-c-betamax} imply that we get $\disp
\s^{k_{0}^{(1)}}\rho^{k_{0}^{(2)}}(2c\beta_{max}x+c\beta_{max}(y^2+x))>d$. This finally yields

\begin{equation}
\label{equ7-cone-instable}
\s^{k_{0}^{(1)}}\rho^{k_{0}^{(2)}}(3y^2+x)>\frac{d}{3c\beta_{max}}.
\end{equation}
This implies the inequality  $\disp
\s^{k_{0}^{(1)}}\rho^{k_{0}^{(2)}}A|u|(3y^{2}+x)>\frac{Ad}{3c\beta_{max}}$, and we  finally get

$$\left|\frac{\lambda^{k_{0}}(b-c\beta)}{\s^{k_{0}^{(1)}}\rho^{k_{0}^{(2)}}Au(3y^2+x)\eps_{1}}
\right|<\frac{(3
c\beta_{max})^2}{Ad^2\eps_{1}}\alpha<\left(\frac{3c\beta_{max}}{d}
\right)^2(3\beta^2+\alpha)\frac1{\eps_{1}A}.$$
Hence, we want to check $\disp \left(\frac{3c\beta_{max}}{d}
\right)^2\frac1{\eps_{1}A}<\frac{c}{2\eps_1}$. By \eqref{condi-majo Abetamax} and \eqref{equ-def-cbeta2} this is equivalent to 
$$\frac{18}{d^2A}<1.$$
Remember that $d<\frac12$ is as close as wanted to $\frac12$, and the desired inequality holds if $A$ is  sufficiently big. As we have seen above, this condition can be realized. Thus,  \eqref{equ-majotermesmoitie} holds.

\paragraph{The case $y> 0$.}
For this case,  \eqref{equ5-cone-instable} yields
$$\s^{k_{0}^{(1)}}\rho^{k_{0}^{(2)}}(bx+c|y|(y^2+x)>\s^{k_{0}^{(1)}}\rho^{k_{0}^{(2)}}|bx-cy(y^2+x)|>d.$$
This inequality  and \eqref{condi-b-c-betamax} imply
$$d<\s^{k_{0}^{(1)}}\rho^{k_{0}^{(2)}}(3c\beta_{max}x+c\beta_{max}y^2)\le
\s^{k_{0}^{(1)}}\rho^{k_{0}^{(2)}}3c\beta_{max}(x+3y^2).$$
This inequality is the same than the one in \eqref{equ7-cone-instable}. Hence,
at that point, sufficient conditions are the same.
\end{proof}

\subsection{Extending the unstable conefield }
\subsubsection{Extending to $\Lambda\setminus\CO(\xi)$}

For a point of the form $M=F(\xi+(x,y))$ (image  by $F$ of the point $\xi+(x,y)$
of the critical region), $\mathsf{C}^{u}(M)$ is already defined.

\paragraph{Notation}
Let $M$ be a point in the critical zone. We write $M=\xi+(x,y)$ and set 
$\disp n_-=n_{-}(M):=\sup\{k,
x\le d\lambda^{k}\}\le +\8$ and $\disp n_+=n_+(M):=\sup\{k,\ \s^{k_1}\rho^{k_2}|bx+cy(y^{2}+x)|\le d\}\le +\8$. 
\begin{definition}
\label{def-crit-tub}
The piece of orbit $F^{-n_{-}}(M),\ldots, M,\ldots, F^{n_{+}}(M)$ is called the {\em critical tube} for $M$. 
\end{definition}

Due to the fact that the forward orbit of $\xi$ never returns into $\mr$, the critical tubes for an orbit are disjoint. Therefore, we can decompose an orbit in critical tubes and ``free'' iterates (see Figure \ref{fig-tubes}).

\begin{figure}[htbp]
\begin{center}
\includegraphics[scale=0.5]{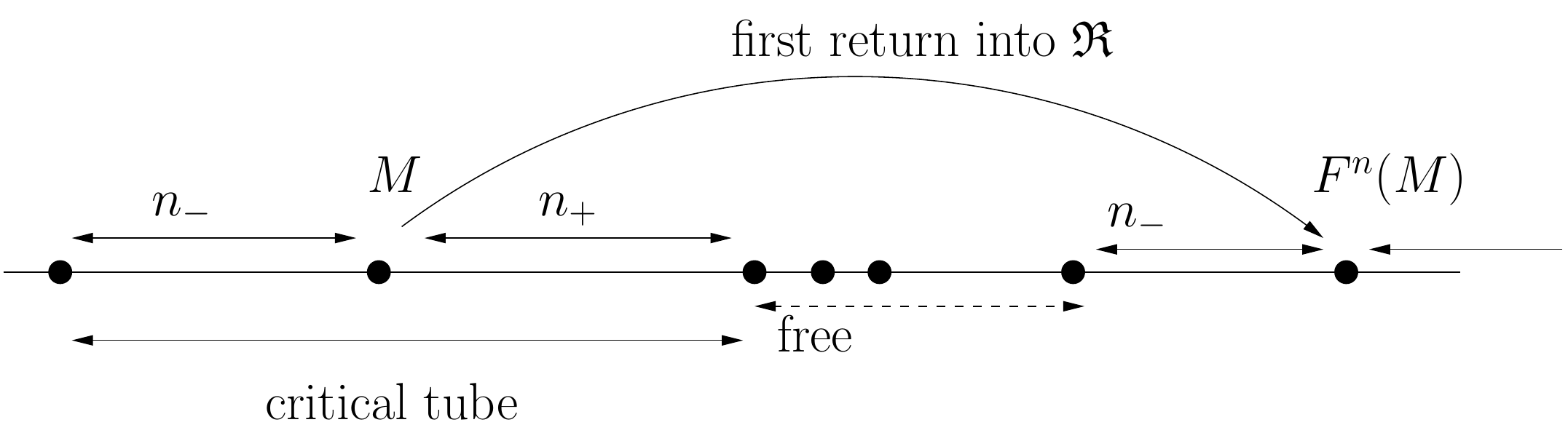}
\caption{Splitting of the orbit in tubes and free moment}
\label{fig-tubes}
\end{center}
\end{figure}

\medskip
We recall that the definition of $\alpha_{max}$ implies that for
$M=\xi+(x,y)$ in the critical zone, $F^{-k}(M)$ belongs to $R_7$ for
$k\in[1;5]$. As we chose $d$ close to $\frac12$, $l_0$ is very small and then, $n_-(M)\ge 5$.

For a point of the form $M=F^{n}(\xi+(x,y))$ (image  by $F^{n}$ of the point
$\xi+(x,y)$ of the critical zone)  with $1\le n\le n_+(M)$, we set
$$\mathsf{C}^{u}(M):=DF^{n-1}(F(\xi+(x,y))).\mathsf{C}^{u}(F(\xi)+\CF(x,y)). $$

For a point of the form $\xi+(x,y)$ we set
$$\mathsf{C}^{u}(\xi+(x,y))=\{\mathbf{v}=(v_{x},v_{y}),\ |v_{y}|\ge
\frac{\s^{n_-}}{\lambda^{n_-}}\frac{Ad}{3c\beta_{max}}|v_{x}|\}.$$
Note that $n_-=+\8$ is possible
if (and only if) $x=0$. In that case the unstable cone is just the vertical
line.

For a point of the form $M=F^{-j}(\xi+(x,y))$ with $\disp 1\le j\le n_-(\xi+(x,y))$, we set
$$\mathsf{C}^{u}(M):=DF^{-j}(\mathsf{C}^{u}(F^{j}(M))).$$

At last, for any other point $M$, we set
$$\mathsf{C}^{u}(\xi+(x,y))=\{\mathbf{v}=(v_{x},v_{y}),\ |v_{y}|\ge
\frac{Ad}{3c\beta_{max}}|v_{x}|\}.$$

\begin{proposition}
\label{prop-cone-instable}
The unstable cone field satisfies, for every point $M\in\Lambda\setminus{\CO(\xi)}$,
$$DF(M).\mathsf{C}^{u}(M)\subset \mathsf{C}^{u}(F(M)).$$
\end{proposition}
\begin{proof}
Some inclusions are direct consequences of the definition: this is the case if
$M$ is of the form $M=F^{-k}(M')$, with $M'$ in the critical zone and $k\le
n_{-}(M')$. This also holds if $M$ is of the form $F^{k+1}(M')$, with $M'$ in
the critical zone and $k\le n_{+} (F(M'))$.

We point out that for $F^{n_{+}}(M)$ the slope of the unstable cone is given by $\disp \frac{\s^{k_{1}}\rho^{k_{2}}}{\lambda^{n_+}}A(3y^{2}+x)$  where $k_{2}$ and $k_{1}$ satisfy $n_{+}=k_{1}+k_{2}$ and are related to the time spent in $R_{1}\cup R_{2}\cup R_{3}$ or other rectangles. It satisfies 
\begin{equation}\label{equ-conunstsortie}
\frac{\s^{k_{1}}\rho^{k_{2}}}{\lambda^{n_+}}A(3y^{2}+x)\ge \frac{Ad}{\lambda^{n_+}3c\beta_{max}}>\frac{Ad}{3c\beta_{max}},
\end{equation}
since we have 
$$d<\s^{k_{1}}\rho^{k_{2}}|bx-cy(y^{2}+x)|<\s^{k_{1}}\rho^{k_{2}}3c\bm(3y^{2}+x).$$

If $M$ is not of that form and neither belongs to the critical zone, then the
result of the proposition is just a consequence of the linearity of the map $F$
at $M$.

If $M=\xi+(x,y)$ is in $\mr$, we recall that the image of vertical vector $(0,1)$ by $DF(M)$ is $-(\eps_1,c(3y^2+x))$ which belongs to the interior of the unstable cone at $F(M)$. Now,
we left it to the reader to check that the same kind of computations than in the proof of
Lemma \ref{lem-cone-instable} show $DF(M).\mathsf{C}^{u}(M)\subset \mathsf{C}^{u}(F(M))$. Just check that we get 
$$|b-cy|\frac{\lambda^{n_{-}}3c\beta_{max}}{\eps_{1}\s^{n_{-}}Ad}<\frac{c}{2\eps_{1}}(3y^{2}+x),$$
since $\disp \left(\frac{3c\beta_{max}}{d}
\right)^2\frac1{\eps_{1}A}<\frac{c}{2\eps_1}$ holds.
\end{proof}

\subsubsection{Extending the unstable conefield to $\CQ\setminus\CO(\xi)$}
We want to point out that the construction of the unstable cone field does not require the point lies in $\Lambda$. Actually, the conefield is firstly define in the post-critical zone, and then extended to the rest of $\Lambda\setminus\CO(\xi)$. The construction only requires that the piece of orbit we consider stays in $\CQ$. We can thus extend the unstable cone fields to every point in $\CQ\setminus\CO(\xi)$ and it will still satisfy 
 $DF(M).\mathsf{C}^{u}(M)\subset \mathsf{C}^{u}(F(M))$
 provided that $F(M)$ belongs to $\CQ$.

\subsection{Expansion and unstable directions}
Here, we prove that vectors in the unstable cone field are expanded by forward iterations of $DF$. We get from this the existence of the unstable direction. 
We recall that for $\mbf{v}:=(v_{x},v_{y})$ we set $||\mbf{v}||_{\8}=\max(|v_{x}|,|v_{y}|)$. 
Let us start with a technical lemma.

\begin{lemma}
\label{lem-hyperbolicite-nin-nout}
Let $M$ be in the critical zone. Let $n<+\8$ and $m<+\8$ be respectively
$n_{-}(M)$ and $n_{+}(M)$. Then for every $\mbf{v}$ in
$\mathsf{C}^{u}(F^{-n}(M))$ we get
$$||DF^{n+1+m}(F^{-n}(M)).\mbf{v}||_{\8}\ge \rho^{\frac{1}{5}(n+1+m)}.$$
\end{lemma}

\begin{proof}
Let us set $M=\xi+(\alpha,\beta)$. Note that our assumptions  $n,m<+\8$ 
yield that neither $\alpha$ nor $\beta$ vanish. 
  Let us pick
$(v_{x},v_{y})=\mbf{v}\neq0$ in $\mathsf{C}^{u}(F^{-n}(M))$. We recall 
$||\mbf{v}||_{\8}=\max(|v_{x}|,|v_{y}|)$.
We also consider the
vector $\mbf{w}=(1,A(3\beta^2+\alpha))$ in $\mathsf{C}^{u}(F(M))$.

By definition of the cone field $\mathsf{C}^{u}$, $|v_{y}|\ge \disp
\frac{Ad}{3c\beta_{max}}|v_{x}|\ge |v_{x}|$, where the last inequality follows
from our choice of the parameters (see Inequality \eqref{equ-grandApour
expansion}). Hence, there is exponential expansion for $\mbf v$ during the piece
of orbit between $F^{-n}(M)$ and $M$ because $DF=DF_{0}$ which expands verticals and contracts horizontals. 

Our strategy is to consider two cases. Either $m$ is ``small'' with respect to
$n$ and the expansion during the entrance phase (from $-n$ to 0) is sufficient
to ensure global exponential expansion between $-n$ and $m$, or $m$ is large,
and then there  also is exponential growth in the unstable direction between
times $0$ and $m$. 

\bigskip
By construction of the unstable cone, the vector $\mbf w$ is more horizontal
than $DF^{n+1}(x).\mbf v$, hence is less expanded (remember that $F\equiv F_{0}$ for these points). Moreover, the vertical
component of $DF^{m}(F(M)).\mbf{w}$ is bigger (in absolute value) than the
horizontal one (its slope is at least $\disp\frac{Ad}{3c\beta_{max}}$ which is
larger than $\disp\frac1{\eps_{1}}$ by Inequality \eqref{equ-grandApour
expansion}).

Therefore, the expansion (between times $F(M)$ and $F^{m}(M)$) of any vector in
the unstable cone $\mathsf{C}^{u}(F(M))$ is larger than $\s^{m}A(3\beta^2+\alpha)\ge
\s^{m}A(\beta^2+\alpha)\ge \s^{m}A\alpha.$ 

On the other hand, the matrix of $DF(M)$ in the canonical basis is 
$$\left(\begin{array}{cc}0 &  -\eps_{1}\\ b-c\beta &
-c(3\beta^2+\alpha)\end{array}\right).$$
By the definition of the map we get 
\begin{equation}
\label{equ1-lem-hyperbo}
\wt{\mbf  v}:=DF^n(F^{-n}(M)).\mbf{v}=(\lambda^{n}v_{x},\s^{n}v_{y}),
\end{equation} 
and the vertical component of this vector is bigger than the horizontal one. 
Therefore we get 
$$||DF(M)\wt{\mbf v}||_{\8}\ge \eps_{1}||\wt{\mbf v}||_{\8},$$
which yields :
\begin{subeqnarray}
\slabel{equ2.1-expansion 0-m}
\text{ on the one hand, }||DF^{m+1}(M)\wt{\mbf v}||_{\8}\ge \rho^{m_{2}}\s^{m_{1}}A\eps_{1}\alpha.||\wt{\mbf v}||_{\8},\\
\mbox{on the other hand, due to \eqref{equ-conunstsortie}, }||DF^{m+1}(M)\wt{\mbf v}||_{\8}\ge
\frac{Ad\eps_{1}}{3c\beta_{max}}.||\slabel{equ2.2-expansion 0-m}\wt{\mbf
v}||_{\8}.
\end{subeqnarray}

\paragraph{First case: $m+1\le 2n$.} We use Inequality \eqref{equ2.2-expansion 0-m}. We
recall that Inequality \eqref{equ-grandApour expansion} means 
$$\frac{Ad\eps_{1}}{3c\beta_{max}}>1,$$
which shows that there is expansion (but not necessarily exponential) in the
unstable direction between $M$ and $F^{m}(M)$. 
Therefore, the expansion for $\mbf{v}$ by $DF^{n+1+m}(F^{-n}(M))$ is larger than
$\disp \s^{n}$ (the expansion between $F^{-n}(M)$ and $M$). 

Hence,  the logarithmic expansion between $F^{-n}(M)$ and $F^{m}(M)$ is larger
than
$$\frac{n}{n+1+m}\log\s\ge \frac{1}{3}\log\s\ge \frac1{5}\log\s\ge \frac15\log\rho.$$

\paragraph{Second case: $m+1>2n$.} We use Inequality \eqref{equ2.1-expansion 0-m}.
Using Inequality \eqref{equ-lambda-sigma}, the expansion between $M$ and
$F^{m}(M)$ in the unstable direction is larger than 
$$\rho^{m}A\eps_{1}\alpha\ge
\rho^{1.2\frac{m}2+0.8\frac{m}2}A\eps_{1}\alpha\ge\rho^{1.2n+0.8\frac{m}2-\frac35}
A\eps_{1}\alpha\ge \frac\alpha{\lambda^n}A\eps_{1}\rho^{0.4m-\frac35}\ge
\rho^{\frac{m+1}5} \frac\alpha{\lambda^n}A\eps_{1}$$
(note that $m\ge 9$ since $n\ge 5$).
Now, we use Inequalities \eqref{equ6-cone-instable} and again
\eqref{equ-grandApour expansion} and \eqref{equ-cbeta-lambda} to conclude that the logarithm expansion
between $F^{-n}(M)$ and $F(M)$ in the unstable direction is larger than 
$$\frac{n+\frac15(m+1)}{n+1+m}\log\rho\ge \frac15\log\rho.$$
 The proof is finished. 
\end{proof}

We can now prove expansion for vectors in the unstable cone filed. 

\begin{proposition}
\label{prop-expansion}
Let $M$ be in $\Lambda\setminus\CO(\xi)$. Then for every $0\neq \mbf{v}\in\mathsf{C}^{u}(M)$, we get
$$\liminf_{\ninf}\frac1n\log||DF^{n}(M).\mbf{v}||_{\8}\ge \frac15\log\rho.$$
\end{proposition}

\begin{proof}
The result is obvious if $M$ returns only a finite number of times in the
critical zone. We thus focus on the difficult case, when $M$ visits infinitely
many times the critical zone by iterations of $F$. For simplicity we assume that
$M$ belongs to the critical zone. Then we denote by $n_{i}$ the time of the
$i^{th}$-visit to the critical zone.
We also denote by $n_{-}(i)$ and $n_{+}(i)$ the integers $n_{-}(F^{n_{i}}(M))$ and $n_{+}(F^{n_{i}}(M))$.

We point out that due to our choice of $d\approx \frac12$ and $l_{0}\approx\frac1\lambda$, $F^{n_{i}-j}(M)$ belongs to $R_{7}$ for $0<j<n_{-}(i)$. On the other hand $F^{n_{i}+j}(M)$ belongs to the same $R_{t}$ than $F^{j}(\xi)$. Since $\xi$ never comes back to $R_{4}\cup R_{7}$ by forward iterates, it is sure that 
$$n_{i}+n_{+}(i)<n_{i+1}-n_{-}(i+1).$$

Lemma \ref{lem-hyperbolicite-nin-nout} shows that for every $i\ge 1$, the logarithmic
expansion between $F^{n_{i}-n_{-}(i)}(M)$ and $F^{n_{i}+1+n_{+}(i)}(M)$ is
bigger than $(n_{-}(i)+1+n_{+}(i))\frac15\log\rho$. 

Moreover, we get for every $i\ge 1$, $n_{i}+1+n_{+}(i)\le
n_{i+1}-n_{-}(i+1)$. By construction of the unstable cone field, the
logarithmic expansion between $F^{n_{i}+1+n_{+}(i)}(M)$ and $F^{
n_{i+1}-n_{-}(i+1)}(M)$ is bigger than $\log\rho$ (considering the norm given by the greatest coordinate in absolute value). 
\end{proof}

Actually, we let the reader check that Lemma \ref{lem-hyperbolicite-nin-nout} and the splitting of orbits as done in the proof of Proposition \ref{prop-expansion} show the next two results. For this we need to introduce a new notion:

\begin{definition}
\label{def-complete}
Let $M$ be in $\disp\bigcap_{k=0}^{n} F^{-k}(\CQ)$. We say that the piece of orbit $M,\ldots, F^{n}(M)$ is \textit{complete} if the time interval $[0,n]$ contains a full number of critical tubes. 
\end{definition}

Then Lemma \ref{lem-hyperbolicite-nin-nout} yields:

\begin{proposition}
\label{prop-expaEucomplete}\label{prop-rem-expan-cu}
Let $M,\ldots, F^{n}(M)$ be a complete piece of orbit. Then, for every $\mbf v$ in $\mathsf{C}^{u}(M)$, 
$$||DF^{n+1}(M).\mbf v ||_{\8}\ge \rho^{\frac{n}5}||\mbf v ||_{\8}.$$
\end{proposition}

We claim that Lemma \ref{lem-hyperbolicite-nin-nout} also yields the next proposition. Its statement is la little bit heavy but it is very usual in the so-called Pesin theory: expansion and contraction are at uniform exponential scale up to a multiplicative term which depends on the point (see Fig. \ref{fig-tipopesinEu}). Note that we do not request the point $M$ to be in the critical zone.

\begin{proposition}
\label{prop-expanEupartielle}
Let $M$ be in a critical tube, say $[-N_{-},N_{+}]$. Then, there exist two positive numbers $l^{+}(M)$ and $l^{-}(M)$, such that \begin{enumerate}
\item for any $n>N_{+}$, if $F^{N_{+}+1}(M),\ldots F^{n}$ is complete, then for every vector $\mbf v$ in $E^{u}(M)$, 
$$||DF^{n}(M).\mbf v ||_{\8}\ge l^{+}(M)\rho^{\frac{n}5}||\mbf v ||_{8}.$$
\item for any $n>N_{-}$, if $F^{-n},\ldots F^{-N_{-}-1}(M)$ is complete, then for every vector $\mbf v$ in $E^{u}(M)$, 
$$||DF^{-n}(M).\mbf v ||_{\8}\le l^{-}(M)\rho^{-\frac{n}5}||\mbf v ||_{8}.$$
\end{enumerate}
Moreover, $l^{+}(M)=l^{+}(M')$ if $M'$ belongs to $\disp B_{N_{+}+1}(M,l_{0})$ and $l^{-}(M)=l^{-}(M')$  if $M'$ belongs to $\disp \bigcap_{k=0}^{N_{-}+1} F^{k}B(F^{-k}(M),l_{0})$.
\end{proposition}

\begin{figure}[htbp]
\begin{center}
\includegraphics[scale=0.55]{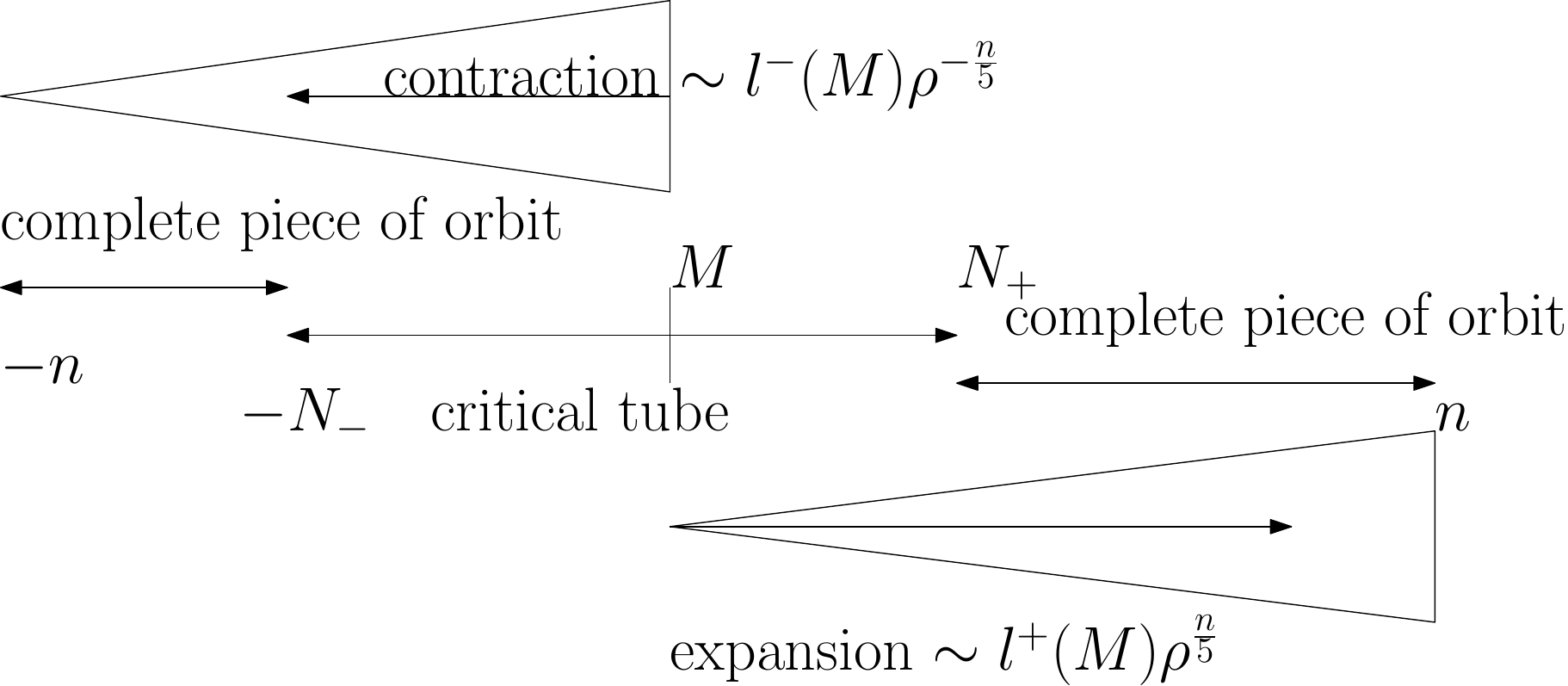}
\caption{Expansions and contractions in the unstable cone.}
\label{fig-tipopesinEu}
\end{center}
\end{figure}

\bigskip
Now, we can now  the existence of the unstable direction.

\begin{proposition}
\label{prop-def-Eu}
For $M\in\Lambda\setminus\CO(\xi)$, $\disp\bigcap_{n\ge 0}DF^{n}(\mathsf{C}^{u}(F^{-n}(M)))$ is a single direction in $\R^{2}$. It is called the unstable direction at $M$ and denoted $E^{u}(M)$. 
\end{proposition}
\begin{proof}
The set $\disp\bigcap_{n\ge 0}DF^{n}(\mathsf{C}^{u}(F^{-n}(M)))$ is a decreasing intersection of cones thus it has a non-empty intersection. If $\mbf{v}$ and $\mbf{w}$ are two normalized vectors in this intersection, we consider two renormalized pre-images $\mbf{v}_{n}$ and $\mbf{w}_{n}$ by $DF^{-n}(M)$. 
We get 
$$|\det DF^{n}|=\frac{|\sin(\measuredangle(\mbf{v},\mbf{w}))|}{|\sin(\measuredangle(\mbf{v}_{n},\mbf{w}_{n}))|}||DF^{n}(\mbf{v}_{n})||.||DF^{n}(\mbf{w}_{n})||.$$

We point out that inequalities \eqref{equ-cbeta-lambda} and \eqref{equ-lambda-sigma}  show that $|\det DF|<1$ everywhere. 
Then 
$$|\sin(\measuredangle(\mbf{v},\mbf{w}))|\le \frac{|\det DF^{n}|}{||DF^{n}(\mbf{v}_{n})||.||DF^{n}(\mbf{w}_{n})||}$$
and Proposition \ref{prop-expanEupartielle} shows that $\measuredangle(\mbf{v},\mbf{w})$ decreases exponentially with  $n$ as $n\to+\8$.  
\end{proof}

\begin{remark}
\label{rem-Eu passe}
Note that $E^{u}(M)$ is actually well defined for every point whose full backward orbit is well defined. 
$\blacksquare$\end{remark}

\subsection{Stable direction}
In this subsection we prove the existence of the stable direction and the exponential contraction of vectors in that direction.

The  stable cone $\mathsf{C}^{s}(M)$ is defined  as the closure of the complement (in $\R^{2}$) of the unstable cone. It is defined in the same domain as the unstable conefield.  It satisfies 
$$DF^{-1}(\mathsf{C}^{s}(M))\subset\mathsf{C}^{s}(F^{-1}(M))$$
provided $F^{-1}(M)$ belongs to $\CQ$.

\begin{proposition}
\label{prop-expand-es}
For every $\disp M\in\bigcap_{n\ge 0}F^{-n}(\CQ)\setminus\CO(\xi)$, $\disp\bigcap_{n\ge 0}DF^{-n}(\mathsf{C}^{s}(F^{n}(M)))$ is a single direction, called the stable direction at $M$ and denoted by $E^{s}(M)$. 

Moreover, let $\disp\Delta:=\max |\det DF|$.  Then there exists some universal constant  $C$ such that for every complete piece of orbit $M,\ldots ,F^{n}(M)$, 
$$|DF^{n}_{|E^{s}(M)}||<C.\Delta^{n}\rho^{-\frac{n}5}.$$
\end{proposition}
\begin{proof}
If there is $n_0$ such that $F^{n}(M)$ does not belong to the critical zone $\forall n>n_0$, then the result follows as in the hyperbolic case,  for $F^{n_0}(M)$, and the unique stable direction for this point is iterated back to define the stable direction for $M$. 

For the rest of the proof we assume that the forward orbit of $M$ visits infinitely many times the critical zone. 
We can thus split this forward orbit in critical tubes and free moments. 
Note that we have seen in the proof of Proposition \ref{prop-expansion}  that there necessarily exists at least one free moment between two consecutive critical tubes.  Then, we can always assume that $M$ belongs to the end of a critical tube and that $F(M)$ belongs to a free moment.

\medskip
We claim that $DF(M)\mathsf{C}^{u}(M)$ is uniformly far from the border of $\mathsf{C}^{u}(F(M))$. 
Actually, at the end of a critical tube Inequality \eqref{equ-conunstsortie} shows that the unstable cone is strictly inside the cones defined by $\disp\{\mathbf{v}=(v_{x},v_{y}),\ |v_{y}|\ge\frac{Ad}{3c\beta_{max}}|v_{x}|\}$. Then its image by the linear map $DF(M)$ is strictly inside $\disp \mathsf{C}^{u}(F(M))=\{\mathbf{v}=(v_{x},v_{y}),\ |v_{y}|\ge\frac{Ad}{3c\beta_{max}}|v_{x}|\}$ and uniformly far from its border.

This shows that if $F^{k}(M)$ is in a free moment, then $\forall\mbf v\in\disp\bigcap_{n\ge 0}DF^{-n}(\mathsf{C}^{s}(F^{n+k}(M)))$ and $\forall\mbf w\in DF.\mathsf{C}^{u}(F^{k-1}(M))$, 
the angle between $\mbf v$ and $\mbf w$ is uniformly bounded from below away from zero. So are they sinus and we call $C$ any positive lower bound. 

Then, consider $\mbf v$ in $\disp\bigcap_{n\ge 0}DF^{-n}(\mathsf{C}^{s}(F^{n+1}(M)))$  and $\mbf w$ in $DF.\mathsf{C}^{u}(F(M))$. Let $\mbf{v}_{n+1}$ and $\mbf{w}_{n+1}$ be  normalized and collinear respectively to $DF^{n+1}(M).\mbf{v}$ and $DF^{n+1}(M).\mbf{w}$.
Again we use the equality 
$$|\det DF^{n}|=\frac{|\sin(\measuredangle(\mbf{v}_{n+1},\mbf{w}_{n+1})|}{|\sin(\measuredangle(\mbf{v},\mbf{w})|}||DF^{n+1}(M)(\mbf{v})||.||DF^{n+1}(M).\mbf{w}||.$$
The left hand side term in exponentially small, $\disp||DF^{n+1}(M).\mbf{w}||$ is exponentially big (due to Proposition \ref{prop-expanEupartielle}), and the two sinus are bounded above by one and below away from zero if $n$ is chosen such that $F^{n+1}(M)$ belongs to a free moment (which happens for infinitely many integers $n$). 

Therefore, for every $n$ and $\mbf v$ as above 
\begin{equation}
\label{equ1-contraCs}
||DF^{n+1}(M)(v)||\le C\Delta^{n}\rho^{-\frac{n}5}||v||,
\end{equation}

We claim that Inequality \eqref{equ1-contraCs} shows that $\disp\bigcap_{n\ge 0}DF^{-n}(\mathsf{C}^{s}(F^{n+1}(M)))$ is a single vector. Otherwise, let us pick $\mbf v$ and $\mbf v'$ in $\disp\bigcap_{n\ge 0}DF^{-n}(\mathsf{C}^{s}(F^{n+1}(M)))$ and adjust them such that $\mbf v'=\mbf v+\mbf w$, with $\mbf w$ in $DF.\mathsf{C}^{u}(F(M))$. Pick $n$ as above. We get 
$$D^{n+1}(F).\mbf v'=DF^{n+1}(M).\mbf v+DF^{n+1}(M).\mbf w.$$
Two of these terms are exponentially small and the last one is exponentially big. This is impossible. 
\end{proof}

\begin{corollary}
\label{cor-map-hyperbolic}
Every ergodic $F$-invariant ergodic measure is hyperbolic: it has one (unstable)
positive  Lyapunov exponent $\lu\ge \frac15\log\rho$ and one (stable) negative
Lyapunov exponent $\ls\le\log\Delta -\lu$.
\end{corollary}
\begin{proof}
If $\mu$ is any ergodic $F$-invariant measure, it is clearly hyperbolic. Moreover 
$$\ls+\lu=\lim_{\ninf}\frac1n\log|DF^{n}(M)|$$
holds $\mu$-a.e, which yields $\ls\le \log\Delta-\lu$. 
\end{proof}

%

\section{Local stable and unstable manifolds}\label{sec-mani}
The goal of this subsection is to construct the local unstable and stable
manifolds. 
The construction has three steps. In the first step we construct the candidate to be the local unstable manifold (see Proposition \ref{prop-var-uns-loc}). In the second step we construct the candidate to be the local stable manifold (see Proposition \ref{prop-varie-stab}). In the third subsection, we prove that they are indeed the local unstable and stable submanifolds. Actually, the main obstacle for the construction is that we have to consider critical tubes in their entire part. This makes more complicated to show that points in the (un)stable manifold are characterized by the fact that the respective distance go to zero  along forward (or backward) orbits.

\subsection{The invariant family of local unstable graphs}
Here we construct an invariant family of local unstable graphs for points which return infinitely many times in $\mr$ by iterations of $F^{-1}$. For other points, their backward orbit eventually stays in the uniformly hyperbolic zone, hence the construction of unstable manifolds is standard. 

We consider in the critical zone a family of vertical curves. Using the local
coordinates, each curve is going to be written as 
$$x=x(y).$$
We define  the family $\CV$ of curves satisfying
\begin{equation}\label{equ-def-cv}
 \left\{\begin{array}{l}
1-\text{ it is a curve joining the top and bottom sides of $\mr$,}\\
2-\quad|x'(y)|\le \disp\frac{1}{6\beta_{max}}(3y^{2}+x(y)),\\
3-\quad|x''(y)|\le D,
\end{array}
\right.
\end{equation}
where $D$ is a fixed positive real number. 

\begin{lemma}
 \label{lem-relvertcurv}
For an open set of parameters and for any
sufficiently big $D$,  the set of graphs $\CV$ is stable by the first return map in the critical zone. 
\end{lemma}
\begin{proof}

Let us consider  $M\in \mr$ be such that its first return into $\mr$ occurs for the $(n+1)^{th}$-iterate. Let $\CC$ be a curve in $\CV$ containing $M$, consider its local parametrization $(x(y),y)$ where $y$ runs $[-\bm,\bm]$ 
and say 
$M=(x(y_{0}),y_{0})$. We want to prove that  the connected component of $F^{n+1}(\CC)\cap \mr$ which contains $F^{n}(M)$ is again in $\CV$.  

The image of the curve  by $F$ is the set
$$\{(-\eps_{1}y,bx(y)-cy(y^{2}+x(y))): y\in [-\bm,\bm]\}.$$
It is a graph over the horizontal interval $[-\eps_1\bm,\eps_1\bm]$, joining the vertical interval, image by $F$ of the bottom side of $\mr$, to the vertical interval, image by   $F$ of the top side of $\mr$. By construction, this graph crosses the rectangle $G_0^{n_c+2,k_c}(F(\xi))$ (see Figure \ref{fig-ima-cv})

\begin{figure}
\begin{center}
\includegraphics[scale=0.6]{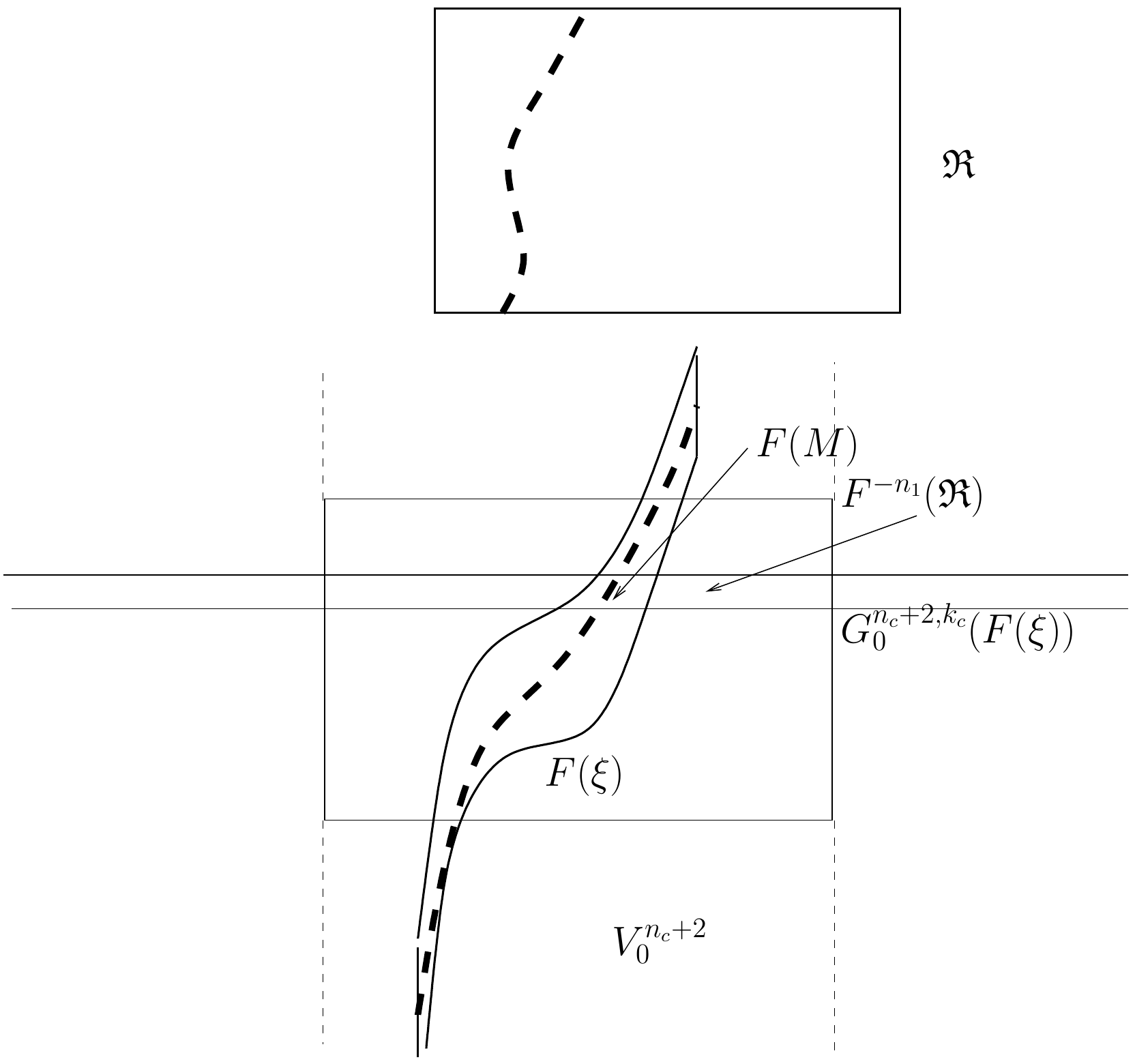}
 \caption{\label{fig-ima-cv}Image of relatively vertical curves}
 \end{center}
\end{figure}

After $n$ iterates of the linear map $F$, this graph is sent to 

$$\{(-\lambda^{n}\eps_{1}y,\s^{k_{1}}\rho^{k_{2}}\left(bx(y)-cy(y^{2}+x(y))\right)):y\in [-\bm,\bm]\}.$$
In particular, it is a curve joining the bottom to the top sides of $\mr$.

By definition, we are considering a curve in the critical zone defined by this
parametrization. 
We set $t=-\lambda^{n}\eps_{1}y$, then $y=-\frac{t}{\lambda^{n}\eps_{1}}$. We use
this new parametrization to compute the derivative and the curvature of this
curve. We want the first coordinated in function of the second but here we have the contrary. 
We recall the formulas for the derivative of an inverse map:
$$\left(\varphi^{-1}\right)'(\varphi(t))=\frac1{\varphi'(t)},\quad
\left(\varphi^{-1}\right)''(\varphi(t))=-\frac{\varphi''(t)}{
\left(\varphi'(t)\right)^{3}}.$$
setting $\varphi(t):=\disp\s^{k_{1}}\rho^{k_{2}}\left(bx(y)-cy(y^{2}+x(y))\right)$, we get 
\begin{eqnarray}
\varphi'(t)&=&-\frac{\s^{k_{1}}\rho^{k_{2}}}{\lambda^{n}\eps_{1}}\left[(b-cy)x'(y)-c(3y^{2}
+x(y))\right]\nonumber\\
\left(\varphi^{-1}\right)'(\varphi(t))&=&\frac{-\lambda^{n}\eps_{1}}{\s^{k_{1}}\rho^{k_{2}}
\left((b-cy)x'(y)-c(3y^{2}+x(y))\right)}\label{equ-deriv-return}\\
\varphi''(t)&=&\frac{-\s^{k_{1}}\rho^{k_{2}}}{\lambda^{n}\eps_{1}}\left[\frac{2c}{\lambda^{n}
\eps_{1}}x'(y)-(b-cy)\frac{x''(y)}{\lambda^{n}\eps_{1}}+\frac{6c}{\lambda^{n}
\eps_{1}}y\right],\nonumber\\
\left(\varphi^{-1}\right)''(\varphi(t))&=&\frac{\s^{k_{1}}\rho^{k_{2}}}{\left(\lambda^{n}\eps_{1
}\right)^{2}}\left[-(b-cy)x''(y)+6cy+2cx'(y)\right]\times\nonumber\\
& & \frac{\left(\lambda^{n}\eps_{1
}\right)^{3}}{\s^{3k_{1}}\rho^{3k_{2}}\left((b-cy)x'(y)-c(3y^{2}+x(y))\right)^{3}}.
\label{equ-curva-return}
\end{eqnarray}
To prove that the curve belongs to the family $\CV$, it remains to check that the
derivative and the curvature respectively satisfy properties $2$ and $3$  in \eqref{equ-def-cv}. As before, we use the local coordinates $(\alpha,\beta)$ for a return
into the critical zone.

\smallskip
$\bullet$ Concerning property $2$ (which deals with the first derivative), 
note that  $\disp (b-cy)x'(y)\le
3c\beta_{max}.\frac{c}{6c\beta_{max}}(3y^{2}+x(y))$. Then, due to
\eqref{equ-deriv-return}, it is sufficient to
check that 
\begin{equation}
\label{equ-control-deriv-return}
\frac{2\lambda^{n}\eps_{1}}{\s^{k_{1}}\rho^{k_{2}}\left(c(3y^{2}+x(y))\right)}\le
\frac{1}{6\beta_{max}}(3\beta^{2}+\alpha)
\end{equation}
holds. Now \eqref{equ-majotermesmoitie} yields 
$\disp \frac{\lambda^{n}c\beta_{max}}{\s^{k_{1}}\rho^{k_{2}}A(3y^{2}+x(y))}\le
\frac{c}2(3\beta^{2}+\alpha)$. We recall Equality \eqref{condi-majo Abetamax} :
$A=\disp\frac{c}{8\eps_{1}}$. 
Thus, 
\begin{eqnarray*}
 \frac{2\lambda^{n}\eps_{1}}{\s^{k_{1}}\rho^{k_{2}}\left(c(3y^{2}+x(y))\right)}&=&\frac{\lambda^{n}c\bm}{\s^{k_{1}}\rho^{k_{2}}A(3y^{2}+x(y))}\frac{2A\eps_1}{c^2\bm}\\
&<&\frac{c}{2}(3\beta^2+\alpha)\frac{2A\eps_1}{c^2\bm}=\frac18(3\beta^2+\alpha),
\end{eqnarray*}
and \eqref{equ-control-deriv-return}
holds. 

\medskip
$\bullet$ Now, let us focus on the control of the curvature. 
To satisfy the condition $3$ on the curvature, \eqref{equ-curva-return} shows it
is sufficient to get 
\begin{equation}
\label{equ-majo-curc-return}
\frac{8\lambda^{n}\eps_{1}}{\s^{2k_{1}}\rho^{2k_{2}}c^{3}(3y^{2}+x(y))^{3}}\left[3c\beta_{max}
D+6c\beta_{max}+\frac{c^{2}}{3c\beta_{max}}(y^{2}+x(y))\right]\le D.
\end{equation}
Inequalities \eqref{equ-lambda-sigma} yield $\disp \lambda<\frac1\s$ and
$\disp\lambda<\frac1{\rho}$, and then it is sufficient to get 
$$\frac{8\eps_{1}}{\s^{3k_{1}}\rho^{3k_{2}}c^{3}(3y^{2}+x(y))^{3}}\left[3c\beta_{max}D+6c\beta_{
max}+\frac{c^{2}}{3c\beta_{max}}(y^{2}+x(y))\right]\le D.$$
Inequality \eqref{equ7-cone-instable} shows that it is sufficient to get
$$\frac{8(3c\beta_{max})^{3}\eps_{1}}{c^{3}d^{3}}\left[3c\beta_{max}D+6c\beta_{
max}+\frac{c^{2}}{3c\beta_{max}}(\beta_{max}^{2}+\alpha_{max})\right]\le D.$$
Now, conditions stated before on the parameters means that the term in the left
hand side  is an affine term in $D$ with slope $\disp \frac{2^3.3^4}{d^3}\bm^2\eps_1$.  As $d$ is very close to $\frac12$, this slope is strictly smaller than 1 if $\bm$ and/or $\eps_1$ are sufficiently small, which is allowed. 
We assume that this holds. Then, 
Depending of the choices of the parameters, there exists $D_{0}$ such that for
$D>D_{0}$, the term in the left hand side is smaller than the one in the right
hand side. 
This shows that \eqref{equ-majo-curc-return} holds if $D$ is chosen sufficiently
big. 
\end{proof}

\begin{remark}
\label{rem-u-deriseconde}
Choosing $\be_{max}$ and $\eps_{1}$ sufficiently small $$D_{0}:=\frac{2^{3}.3.7(19\be_{max}^{2}+\al_{max})}{d^{3}}\frac{\eps_{1}}{1-\frac{2^3.3^4}{d^3}\bm^2\eps_1},$$
is as closed as wanted to 0. 
$\blacksquare$\end{remark}

\begin{proposition}
 \label{prop-var-uns-loc}
Let us consider a point $M\in\mr$ returning infinitely
many times in the critical zone by iterations of $F^{-1}$. Let $n_{k}$, $k=1,2,\ldots$ be these return times in $\mr$. Then, for each $k$, there exists a curve $\CC_{k}^{u}$ in $\CV$ containing $F^{-n_{k}(M)}$ and such that for every $k<j$
$$F^{-(n_{j}-n_{k})}(\CC^{u}_{k})\subset \CC^{u}_{j}.$$
\end{proposition}
\begin{proof}
 Let us consider a point $M$ close to the critical value and returning infinitely
many times in the critical zone by iterations of $F^{-1}$. Let $(n_{k})_{k\ge 1}$ be the
sequence of positive integers such that $M_k:=F^{-n_{k}}(M)$ belongs to critical zone (and $n_{0}:=0$).

The two vertical sides of $\mr$ are curves in $\CV$, and the  point $M_k$is between them. Following the proof of Lemma \ref{lem-relvertcurv}, $M_{k-1}$ is between two curves in $\CV$ which are respectively images by $F^{n_k-n_{k-1}}$ of these two vertical segments. Moreover, these two curves are also between the two sides of $\mr$.

Doing this by induction, we get two families of curves, $(\CC_{k,l})_k$  and $(\CC_{k,r})_k$,  images by $F^{n_k}$ of the vertical sides of $\mr$, and such that $\CC_{k,l}$ is always on the left of $M$ and $\CC_{k,r}$ is always on the right of $M$. 

Taking orientation from the left to the right, the family of curves  $(\CC_{k,l})_k$ is an increasing family -- each one has a parametrization $x_{k,l}(y)$, with $x_{k+1,l}>x_{k,l}$ -- and the family $(\CC_{k,r})_k$ is an decreasing family. 

Lemma \ref{lem-relvertcurv} shows that the limit curves $\CC_{\8,l}$ and $\CC_{\8,r}$ are in $\CV$. We prove  by contradiction that the two limit curves are actually equal. 
If not, Lemma \ref{lem-relvertcurv} shows that there must be some ``bubbles'' between
these two curves (see Figure \ref{fig-bubble}). Let us denote such a bubble by $\CB$. This bubble has positive
Lebesgue measure. By construction, for every point  $M'$ in the bubble  and for
every integer $n$, $F^{-n}(M')$ belongs to the ball $B(F^{-n}(M),1)$.

\begin{figure}

\begin{center}

\includegraphics[scale=0.7]{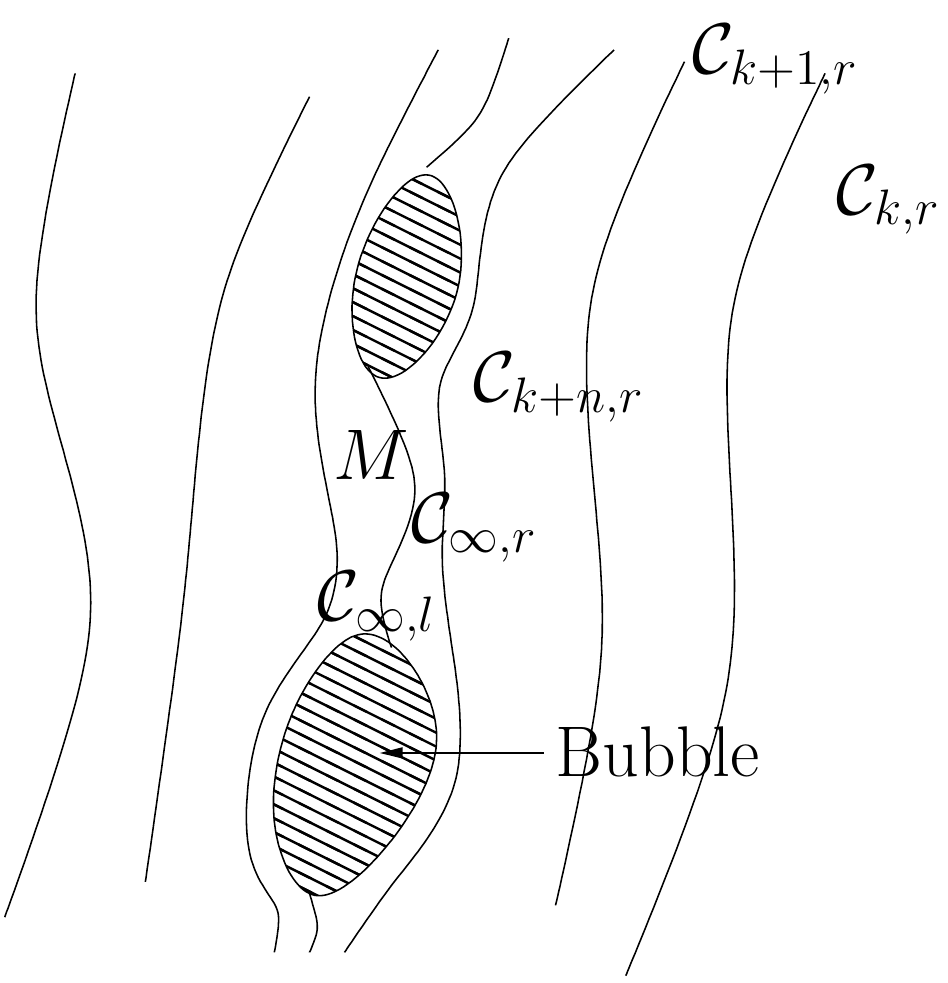}
 \caption{\label{fig-bubble}Family of curves accumulating to unstable manifold.}

\end{center}

\end{figure}

On the other hand, $|\det DF|<1$, then the volume of $F^{-n}(\CB)$ goes to
$+\8$. This produces a contradiction and proves that there cannot be bubbles
between the two curves. In other words, the two curves coincide. This defines the curve $\CC^{u}_{0}$. 

Doing the same for each $M_{k}$, we get a family of curves $\CC^{u}_{k}$ in $\CV$. By construction, if $k<j$, 
$$F^{-(n_{j}-n_{k})}(\CC^{u}_{k})\subset \CC^{u}_{j}.$$
\end{proof}

\subsection{The invariant family of local stable graphs}

Now, we construct a family of local stable manifold in the post-critical
zone. 
We prove they are uniformly long and relatively horizontal curves. The
construction holds only for points which return (forward) infinitely many often
in the critical zone. For other points, the construction of the stable manifold
is standard and follows from the fact that these points are (forward) uniformly
hyperbolic.

\paragraph{Notation:}
 Let $M$ be a point in $\mr$ whose first return in $\mr$ is $n+1$. We denote by $F_{F(M)}^{-n}(\mr)$ the connected component of $F^{-n}(\mr)\cap \CQ$ which contains $F(M)$. It is a horizontal band in $\CQ$ (note that  shadowing $M$,$F^{-n}$ is linear).

\begin{definition}\label{def-relhoricurv}
 Let $M$ be a point in $\mr$ whose first return in $\mr$ is $n+1$.  A curve containing $F(M)$ is said to be an approximative stable curve if it is a $\CC^{1}$-graph over the whole interval $[0,1]$ in $F_{F(M)}^{-n}(\mr)$ with slope at every point in the stable cone field (when it makes sense).
\end{definition}

\begin{lemma}\label{lem-preim-hori-curv}
Let $M$ be a point in $\mr$. Assume that $n+1$ is its first return time in $\mr$ for $F$. If $\CC$ is an approximative stable curve containing $F^{n+2}(M)$, then the connected component in $\CQ$ of $F^{-n-1}(\CC)$ which contains $F(M)$ is again an approximative stable curve. 
\end{lemma}
\begin{proof}
 Assume that the curve has a parametrization which has the form $(\eta,\zeta(\eta))$ in $G_0^{n_{c}+2,k_{c}}(F(\xi))$. 
This curve must intersect the two cubics, image by $F$ of the left hand and right hand sides of $\mr$. This proves that  the connected component of $F^{-1}(\CC)$  which contains $F^{n+1}(M)$ is a $\CC^{1}$-curve joining the left hand side to the right hand side of $\mr$. By construction, the tangent direction to this curve  at each point is contained in the stable cone-field at this point. Applying $F^{-n}$, we get in  $F^{-n}_{F(M)}(\mr)$ a $\CC^{1}$-curve joining the left hand side of $\CQ$ to it right hand side (note that $n>k_c$ because $F(M)$ is in $F(\mr)$), still with tangent direction in the stable cone fields.  
By definition, the slope of the tangent lines are bounded from above, and the curve must be a graph. By construction, it is a graph over the whole interval $[0,1]$ contained in the band $F^{-n}_{F(M)}(\mr)$. 
\end{proof}

\begin{proposition}
\label{prop-varie-stab}
Let $M$ be in $\mr$ and returning infinitely many often in $\mr$ by iterations of $F$. Let $n_{k}$, $k=1,2,\ldots$ be the positive return times, $n_{0}:=0$. Then, for each $k$ there exists  a curve  $\CC_{k}^{s}$ containing $F^{1+n_{k}}(M)$  and with slope in the stable cone field such that  for every $k<j$, 
$$F^{n_{j}-n_{k}}(\CC^{s}_{k})\subset \CC_{j}^{s}.$$
\end{proposition}
\begin{proof}
We consider a family of positive integers $n_{k}$, such that $n_{1}+1+n_{2}+1+\ldots +n_{k}+1$ is the $k^{th}$ return time in $\mr$. 
For simplicity we denote by $M_{k}$ the $k^{th}$ return of $M$ in $\mr$. Thus we have 
$$M_{k+1}=F^{n_{k+1}}(F(M_{k})),$$
and $F(M_{k})$ belongs to $F(\mr)$. 

We construct a family of curves which accumulate on the local stable manifold. 
Consider $k$ big. $F(M_{k})$ belongs to $F(\mr)$, and $F(\mr)$ crosses the horizontal band $F^{-n_{k+1}}_{F(M_{k})}(\mr)$ (see Figure \ref{fig-longcubics}). 
 
\begin{figure}
\begin{center}
\includegraphics[scale=0.5]{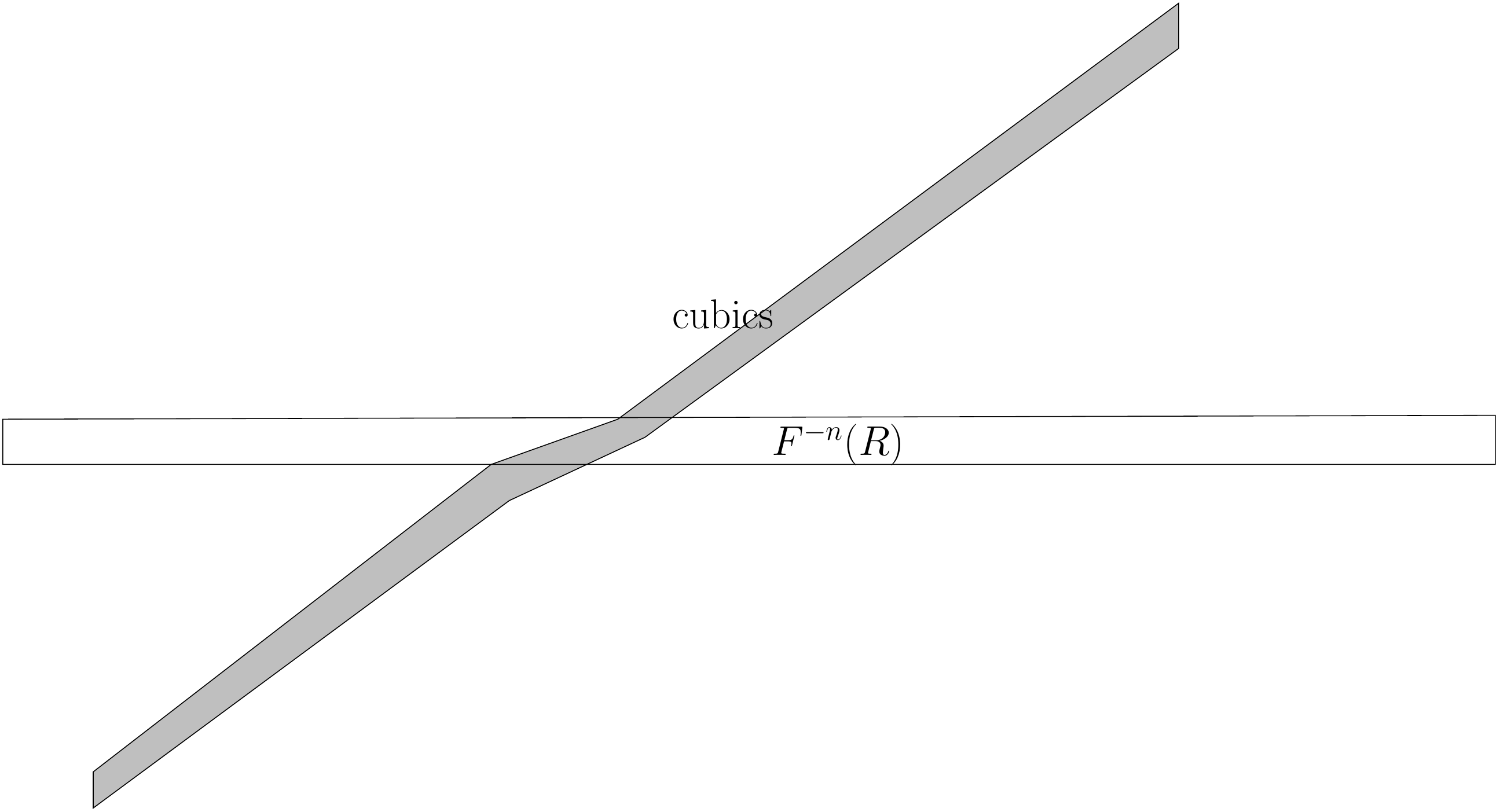}
 \caption{Long $u$-crossing}\label{fig-longcubics}
\end{center}
\end{figure}

Let us consider two  vertical segments. The first one starts at the intersection of the top horizontal line of $F^{-n_{k+1}}_{F(M_{k})}(\mr)$  with the right hand side cubic, and goes upwards until it reaches the left hand side cubic. 
The second one starts at the intersection of the bottom horizontal line of $F^{-n_{k+1}}_{F(M_{k})}(\mr)$   with the left hand side cubic, and go downwards until it reaches the right hand side cubic (see  Figure \ref{fig-doublefin}). 
The region bounded by these two segments and the cubic curves will be referred to as the zone with double fins.

\begin{figure}[htb]

\begin{center}
\includegraphics[scale=0.4]{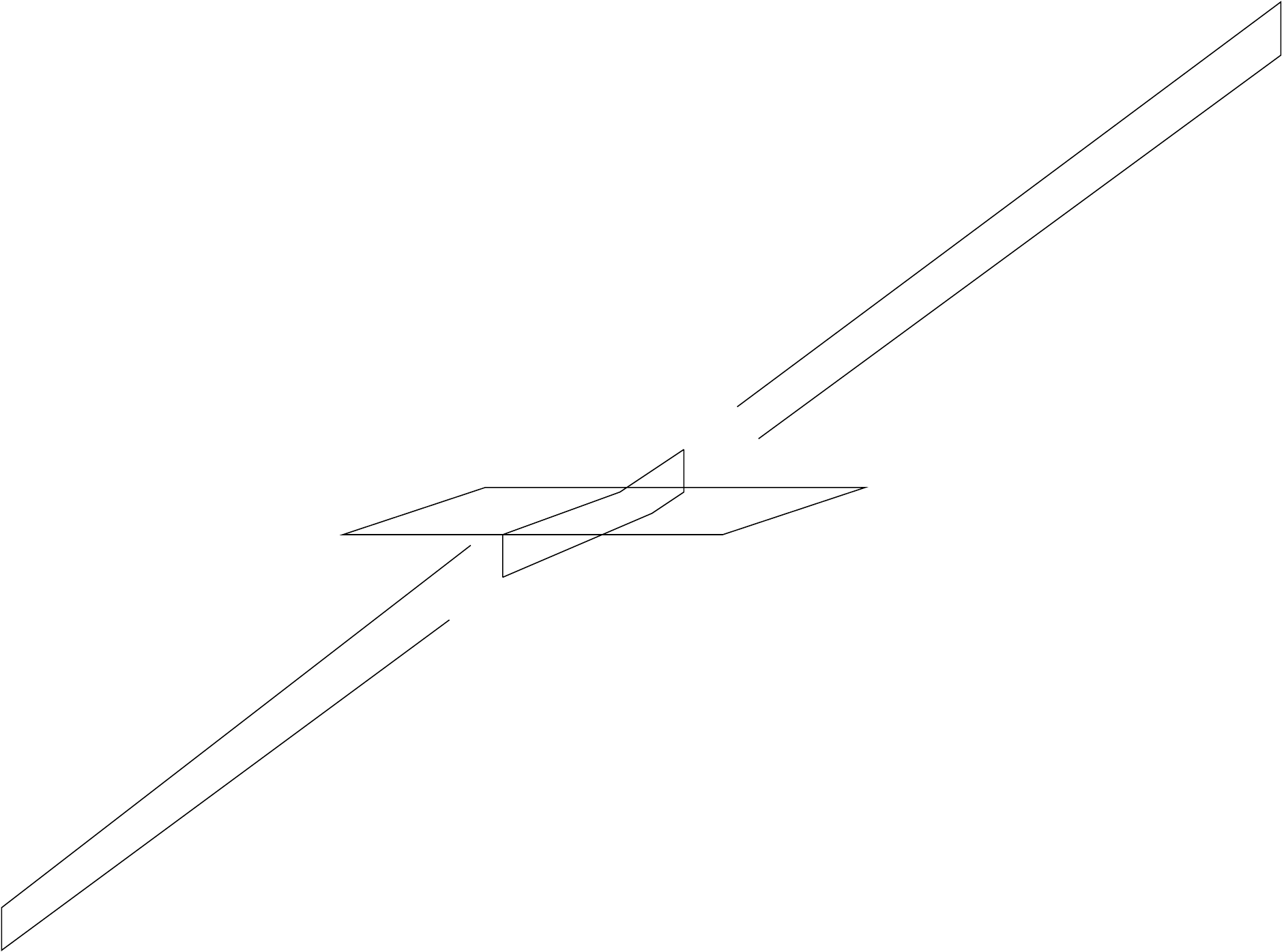}
 \caption{The two added fins}\label{fig-doublefin}
\end{center}

\end{figure}

 By construction, this zone is entirely in $F(\mr)$ and contains $F(M_{k})$. Its preimage by $F$ gives a rectangle $R_{k}$: the two vertical segments are sent over horizontal ones, and the two cubics are sent over vertical segments (see Figure \ref{fig-preimdoublefin})
\begin{figure}
\begin{center}
\includegraphics[scale=0.5]{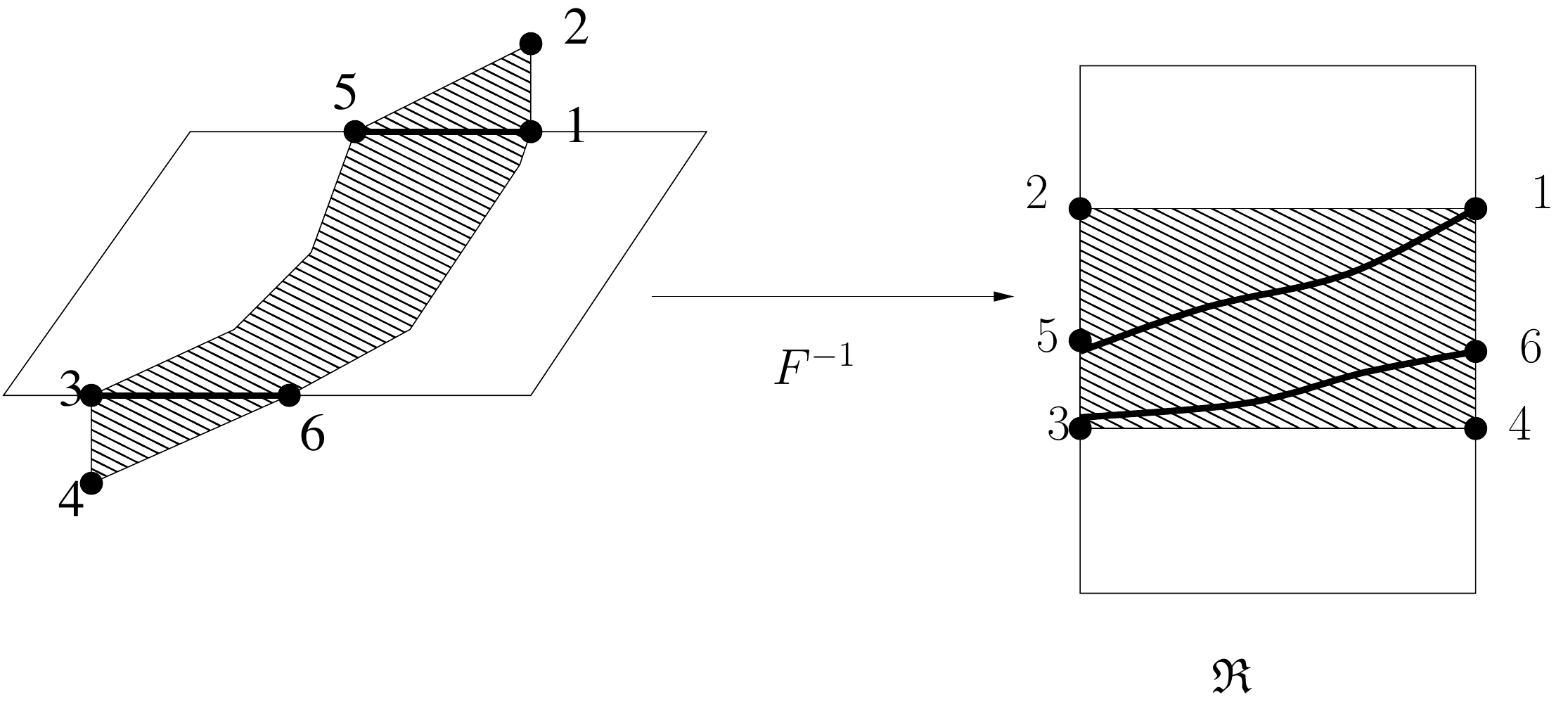}
\caption{Preimage of double fins zone}\label{fig-preimdoublefin}
\end{center}
\end{figure}

This rectangle has the same horizontal length than $\mr$ (namely $\am$) but has smaller height and is inside $\mr$. It also contains $M_{k}$. 
The image by $F^{-1}$ of the top and bottom sides of $F^{-n}\mr$ intersected to the zone with double fins, are curves joining one edge of the rectangle with the border of the other vertical side (see Figure \ref{fig-preimdoublefin}). Using Cardan's method to solve degree-3 equations we see that the curves are actually  graphs of a $\CC^{1}$ map over the horizontal side of the rectangle $R_{k}$, with slope in the stable cone. 
Lemma \ref{lem-preim-hori-curv} shows that their components by the push-backward $F^{-n_{k}}$ in $F^{-n_{k}}_{F(M_{k-1})}(\mr)$ are two approximative stable curves in $F^{-n_{k}}_{F(M_{k-1})}(\mr)$. 
 We can reproduce the construction of the double fins but with this thinner rectangle (see Figure \ref{fig-toto1}).

\begin{figure}
\begin{center}
\includegraphics[scale=0.5]{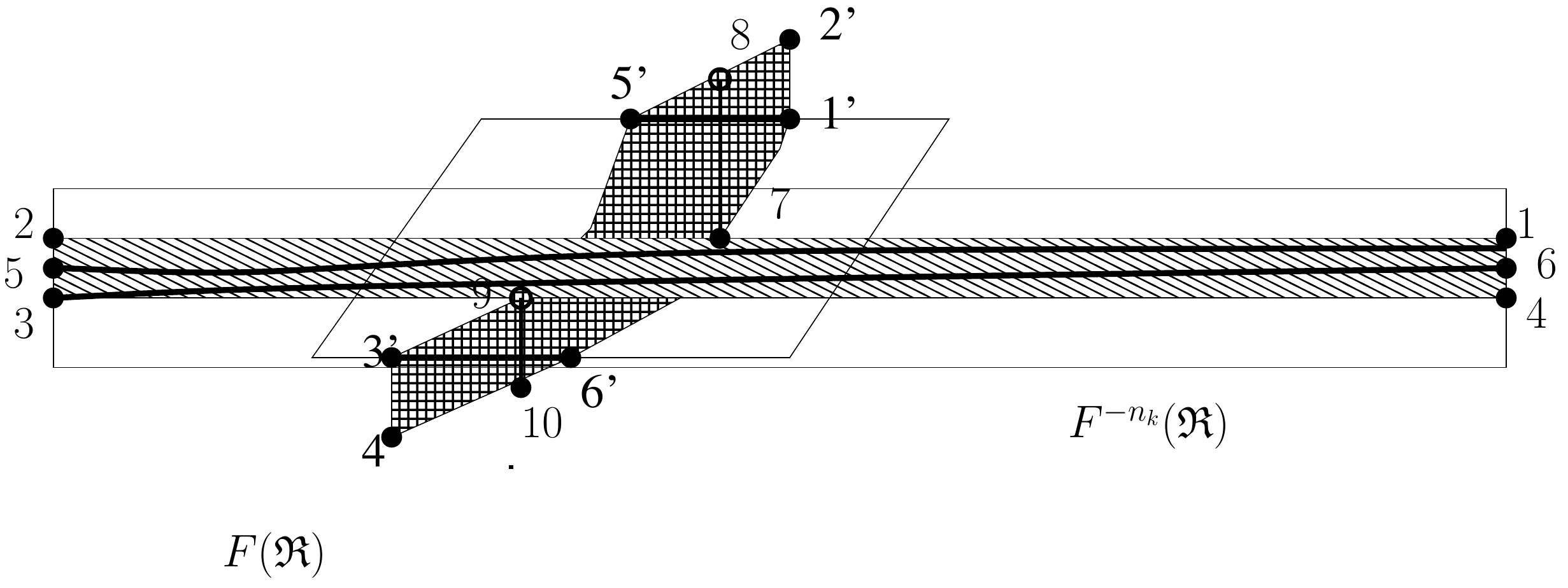}
 \caption{Second  and thinner double fins zone for thinner rectangle}
 \label{fig-toto1}
\end{center}
\end{figure}

The image by $F^{-1}$ of these curves yields 4 curves with slope in the stable cone in $\mr$ as described in Figure \ref{fig-preim5doubelfin}.

Then, we get in $F^{-n_{k-1}}_{F(M_{k-2})}(\mr)$ 4 approximative stable curves, and $F(M_{k_{1}})$ is below two and above two of these curves. The two extremal curves are the one obtained by doing the construction we described (double fin structure) but starting from the point $F(M_{k-1})$, and the two intermediate curves are the one obtained by our first construction (and pushed backward).

Doing this construction by induction, we get $F^{-n_{1}}_{F(M)}(\mr)$ a sequence of decreasing graphs and a sequence of increasing graphs, say $\CC_{k,+}$ and $\CC_{k,-}$. The point $F(M)$ is above every curve $\CC_{k,-}$ and below every $\CC_{k,+}$. Each curve $\CC_{k,+}$ or $\CC_{k,-}$ is a preimage of the horizontal part of one of the fins close to $F(M_{k})$. 

Again, Lemma \ref{lem-preim-hori-curv} shows that all these curves are approximative stable curves, thus have bounded slope. In other words, we have an increasing and a decreasing sequence of equicontiuous functions. Ascoli's theorem shows that both converge  to two curves, say $\CC_{\8,+}$ and $\CC_{\8,-}$. Their slopes are in the stable cones.

\begin{figure}[htb]
\begin{center}
\includegraphics[scale=0.5]{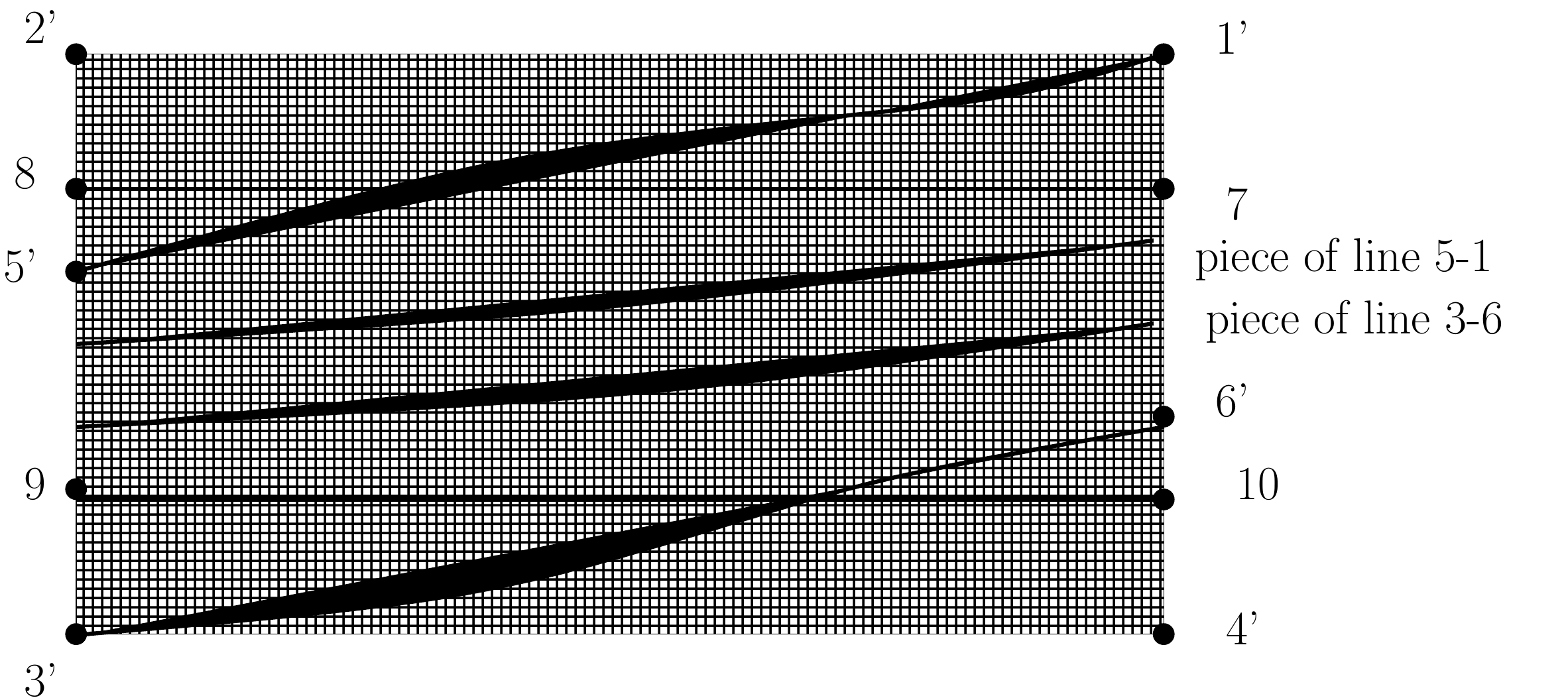}
\caption{4 curves in $\mr$}
\label{fig-preim5doubelfin}
\end{center}
\end{figure}

To finish the proof, we have to prove that these two curves are the same.
We copy and adapt the proof for the unstable manifold. The two curves are continuous, thus, if they do not coincide, the set where they are different are open and this generates ``bubbles''. 

Assume that there is a bubble between these two curves.

Pushing forward,  the image of the bubble is approaching the forward orbit of $M$ since there is exponential contraction in the stable direction (see Proposition \ref{prop-expand-es}). Moreover, by construction, the forward orbit of the bubble  stays in $\CQ$ and points in it have a well-defined unstable cone over them. 
We can thus take into the initial bubble a curve with slope in the unstable cones. By iterating forward, the length of the images of this curves grows exponentially fast (see Prop. \ref{prop-expanEupartielle}). This contradicts the fact that for all the points in the bubble, their forward image is always very close to the forward images of $M$.

Let $\CC_{\8,+}=\CC^{s}_{0}$ and consider, for each $F^{n_{k}}(M)$ a curve $\CC^{s}_{k}$ constructed the same way, each one being an approximative stable curve containing $F^{1+n_{k}}(M)$.  By construction, if $k<j$, 
$$F^{n_{j}-n_{k}}(\CC^{s}_{k})\subset \CC^{s}_{j}.$$
\end{proof}

We emphasize that at that step, the curves defined in Proposition \ref{prop-varie-stab} are just Lipschitz graph over the interval $[0,1]$, as they are obtained as limit of $\CC^{1}$ graphs with bounded slope. 
\subsection{Local (un)stable manifolds and product structure}

Consider a point $M$ in $\Lambda\cap \mr$. Assume it returns to $\mr$ infinitely many often in the past and in the future. Let $n_{k}^{-}$ and $n_{k}^{+}$ be the $k^{th}$-return times for backward and forward iterates into $\mr$. 
By Proposition \ref{prop-var-uns-loc} there exists a local manifold $\CC_{0}^{u}$ such that for every $M'$ in this manifold, the distance between $F^{-n^{-}_{k}}(M')$ and $F^{-n^{-}_{k}}(M)$ converges to $0$. 

Similarly, Proposition \ref{prop-varie-stab} yields a local stable manifold $\CC^{s}_{0}$, such that for every $F(M')$ in it, the distance between $F^{n_{k}^{+}+1}(M')$ and $F^{n_{k}^{+}+1}(M)$ converges to $0$.

By construction we clearly have for $M'$ in $\CC^{u}_{0}$, $\liminf_{\ninf}|F^{-n}(M)-F^{-n}(M')|=0$ and $M''$ in $\CC^{s}_{0}$, $\liminf_{\ninf}|F^{n}(M)-F^{n}(M'')|=0$

Nevertheless, we are interested in a converse  and more general result : if $M'$ satisfies that $|F^{-n}(M')-F^{-n}(M)|\to_{\ninf}0$, does this imply that for some $n$, $F^{n}(M')\in C^{u}_{0}(F^{n}(M))$? This is the purpose of Proposition \ref{prop-prodstruc} and the definition of the local product structure which separates points.

\subsubsection{Definition of $W^{u,s}_{loc}(M)$}
We start by recalling that  the construction of (un)stable local manifolds $W^{u,s}_{loc}(M)$ is standard for points which returns only finitely many times in $\mr$. We recall that the stable manifold is defined as soon as the forward orbits stays in $\CQ$ and the unstable manifold is defined as soon as the backward orbits stays in $\CQ$. 

\medskip
For $\xi$, we recall that its unstable manifold is naturally defined.It is also the unstable manifold for $(0,0)$ in $\CQ$ (bottom left point). 
Its stable manifold is also well defined, actually the stable local manifold for $F(\xi)$ is the horizontal line which contains $F(\xi)$ in $\CQ$. Then $W^{u,s}_{loc}(M)$ is well defined for any $M$ in $\CO(\xi)$.

\medskip
We thus now focus on points which return infinitely many times in $\mr$.

\begin{definition}
\label{def-ws}
Let $M$ be a point that returns infinitely many often in $\mr$ by iteration of $F$. Let $n$ be the smallest non-negative  integer such that $F^{n}(M)$ belongs to $\mr$, and let $\CC^{s}(F^{n+1}(M))$ be the curve constructed in Prop. \ref{prop-varie-stab} and containing $F^{n+1}(M)$. 

The connected component of $F^{-n-1}(\CC^{s}(F^{n+1}(M)))$ in $\CQ$ which contains $M$ is called the \emph{local stable manifold of $M$} and is denoted $W^{s}_{loc}(M)$. 
\end{definition}
We emphasize that, by construction, $W^{s}_{loc}(M)$ is a graph over $[0,1]$ and is inside the horizontal band $H_{0}^{1}(M)$. Moreover its slope at any point is in the stable cone at this point. We have also seen it is a $\CC^{1+}$ curve. 

%

\bigskip
Now, we construct long ``vertical'' local unstable manifolds. Let $M$ be in the critical zone $\mr$, and returning infinitely many often in $\mr$ by iteration of $F^{-1}$. By construction, the curve $\CC^{u}(M)$ given by Proposition \ref{prop-var-uns-loc} is a ``vertical'' graph  going from the bottom of $\mr$ to the top of $\mr$. Its image by $F$ is a graph going from the vertical segment on the left hand side of $F(\mr)$ to the vertical segment of the right hand side of $F(\mr)$ (see Figure \ref{fig-ima-cv}).
This curve crosses the horizontal band $H_{0}^{k_{c}}(F_{0}(\xi))$, from the bottom to the top. 

All the points in this band have the same itinerary than  $F_{0}(\xi)$ by iteration of $F_{0}$ for the next $k_{c}$ iterations. By assumption, $F_{0}(\xi)$ never comes back to the rectangles $R_{7}$ and $R_{4}$. This means that for all these points, their images by $F^{j}$ and by $F_{0}^{j}$ coincide for $j\le k_{c}$. Hence,   $F_{k_{c}+1}(\CC^{u}(M))$ is a curve joining the bottom of the square $\CQ$ to its top. 

Moreover, this curve is a graph over the vertical segment $[0,1]$ (the left side of $\CQ$) with slope in the unstable cone field. By construction it is contained in a vertical band $V_{0}^{1}$. 

\begin{definition}
\label{def-wu}
Let $M$ be a point which returns infinitely many often in $\mr$ by iteration of $F^{-1}$. Let $n$ be the smallest positive integer such that $F^{-n}(M)\in \mr$. The connected component of $F^{n}(\CC^{u}(F^{-n}(M)))$  in $\CQ$ which contains $M$ is called the \emph{local unstable manifold} of $M$. It is denoted by $W^{u}_{loc}(M)$
\end{definition}
We have just seen above that $W^{u}_{loc}(M)$ is a graph over the vertical segment left side of $\CQ$, with slope in the unstable cone field, going from the bottom to the top of $\CQ$, and included in the vertical band $V_{0}^{1}(M)$. 

%

\begin{proposition}
\label{prop-EuEsWus}
For $M\in\Lambda\setminus\CO(\xi)$, $T_{M}W^{u,s}_{loc}(M)=E^{u,s}(M)$.
\end{proposition}
\begin{proof}
We recall that unstable and stable conefields may be extended to points whose forward \ul{or} backward orbit stays in $\CQ$ (a weaker condition than being in $\Lambda$). By construction, this holds for points in any $W^{u,s}_{loc}$ and not only for point in $\Lambda$.

More precisely, for any point $M'\in W^{u}_{loc}(M)$, the backward orbit  of $M'$ is well defined and so is the unstable conefield. We can consider $\disp\cap_{n}DF^{n}(\mathsf{C}^{u}(F^{-n}(M')))$; this is a single direction $E^{u}(M')$ (see Prop. \ref{prop-def-Eu}). 

The proof of Proposition \ref{prop-def-Eu} also shows that the tangent space to $W^{u}_{loc}(M)$ at $M'$ is  $E^{u}(M')$.

A similar result holds for $M''\in W^{s}_{loc}(M)$ and $E^{s}(M'')$. 
\end{proof}

An immediate consequence of Propositions \ref{prop-expansion} and \ref{prop-EuEsWus} and Corollary \ref{cor-map-hyperbolic} is that for $M'$ in $W^{u}_{loc}(M)$, 
$$d(F^{-n}(M),F^{-n}(M'))\to_{\ninf}0,$$
and for any $M''\in W^{s}_{loc}(M)$, 
$$d(F^{n}(M),F^{n}(M'))\to_{\ninf}0.$$

\subsubsection{The local product structure}

Roughly speaking, Proposition \ref{prop-EuEsWus} means that $F$ expands unstable manifolds and $F^{-1}$ expands stable manifolds. This, combined with the fact that the graphs are transversal, yields a local product structure which separates points:

\begin{proposition}
\label{prop-prodstruc}
Let $M$ and $M'$ be in $\Lambda$. Then, $W^{u}_{loc}(M)\cap W^{s}_{loc}(M')$ is a single point denoted by $[M,M']$, which is in $\Lambda$. Moreover,

\begin{itemize}

\item The point $M'$ satisfies $\disp|F^{-n}(M)-F^{-n}(M')|\to_{\ninf}0$ if and only if there exists $m$ such that $F^{m}(M')\in W^{u}_{loc}(F^{m}(M))$.

\item The point $M'$ satisfies $\disp|F^{n}(M)-F^{n}(M')|\to_{\ninf}0$ if and only if there exists $m$ such that $F^{m}(M')\in W^{s}_{loc}(F^{m}(M))$.
\end{itemize}
\end{proposition}
\begin{proof}
The fact that $W^{u}_{loc}(M)\cap W^{s}_{loc}(M')$ is non empty  is a consequence of the fact that these two graphs cross $\CQ$ in all its extension, the first from the bottom to the top, and the second from the left to the right side of $\CQ$. 

The intersection is a single point because  at any intersection point, the slope  of $W^{u}_{loc}$ is always strictly more vertical than the slope of $W^{s}_{loc}$, except for the countable set of points $\{ F^{k}(\xi),k\in \Z\}$ where both directions coincide.  

By construction, the point $M'':=[M,M']$ has its backward orbit converging to the backward orbit of $M$, and its forward orbit converging to the forward orbit of $M'$, both being always in $\CQ$, and then it belongs to $\Lambda$. 

\medskip
Let us now prove the two last properties. For that, we assume that the orbit (backward or forward) of $M$ visits infinitely many often $\mr$, otherwise the proof is as in the uniformly hyperbolic setting\footnote{Even if $M=\xi$.}. 

We assume just for now that $M$ belongs to $\mr$. Let $n_{k}$, $k=1,2,\ldots$ be the sequence of positive return times into $\mr$, and pick $P\in \CC^{u}(M)$, $P\neq M$. 

\paragraph{Claim.} There exits a positive integer $k$ such that $F^{n_{k}}(P)\in \CC^{u}(F^{n_{k}}(M))$ but $F^{n_{k+1}}(P)\notin \CC^{u}(F^{n_{k+1}}(M))$. 

\begin{proofof}{the claim}
By Proposition \ref{prop-rem-expan-cu}, for every $j$, $DF^{-n_{j}}$
acts as a contraction  of ratio at least $(l^{+}(M))^{-1}\rho^{-\frac{n_{j}}5}$ in $TW^{u}_{loc}(F^{n_{j}}(M))$. 
\end{proofof}

Let us consider this biggest integer $n_{k}$. 

\paragraph{Claim.} Let $n_{k}$ be as above. There exists $n>n_{k}$ such that $F^{n}(M)$ and $F^{n}(P)$ belongs to two different horizontal bands $H_{0}^{1}$. 

\begin{proofof}{the claim}
The two points $F^{n_{k}+1}(P)$ and $F^{n_{k}+1}(M)$ belongs to $F(\mr)$. If they have two different first return times to $R_{4}$, then either $F^{n_{k+1}}(P)$ does not belong to $H_{0}^{1}(\xi)$ (and the claim is proved) or it belongs to $H_{0}^{1}(\xi)$ but then it does not belong to $V_{0}^{1}(\xi)$ (namely it does not belong to $R_{1}\cup R_{4}\cup R_{7}$). Then, $F^{n_{k+1}-1}(M)$ and $F^{n_{k+1}-1}(P)$ belong to two different horizontal bands $H_{0}^{1}$. 

If $F^{n_{k+1}}(P)$ belongs to $R_{4}$, it cannot belong to $\mr$. If it does not belong to $H_{0}^{k_{c}+1}$, then $F^{n_{k+1}+k_{c}}(M)$ and $F^{n_{k+1}+k_{c}}(P)$ belong to two different horizontal bands $H_{0}^{1}$ (note that $F$ and $F_{0}$ coincide for times between $n_{k+1}+1$ to $n_{k+1}+k_{c}$).

Finally, if $F^{n_{k+1}}(P)$ belongs to $H_{0}^{k_{c}+1}(\xi)$, then it does not belong to the same vertical band $V_{0}^{n_{c}+1}$. By construction, the first return time to $\mr$ of a point $\mr$ is bigger than $n_{c}+k_{c}+2$,and  this shows that there exists $n_{k}<n<n_{k+1}$ such that $F^{n}(M)$ and $F^{n}(P)$ belong to two different vertical bands $V_{0}^{1}$. Then, $F^{n-1}(M)$ and $F^{n-1}(P)$ belong to two different horizontal bands $H_{0}^{1}$.
\end{proofof}

\bigskip
Let us now finish the proof of  Proposition \ref{prop-prodstruc}. Assume that $M$ and $M'$ satisfy 
$$|F^{n}(M')-F^{n}(M)|\rightarrow_{\ninf}0.$$
Let consider $n_{d}$ such that for every $n>n_{d}$, 
$$|F^{n}(M')-F^{n}(M)|<d,$$
and consider $P:=[F^{n}(M),F^{n}(M')]$. By construction, for every $j\ge 0$, $F^{j}(P)$ and $F^{n+j}(M')$ belong to the same horizontal band $H_{0}^{1}$. The two claims show that if  $P\neq M$, then for some $j$, $F^{j}(P)$ and $F^{n+j}(M)$ belong to two different horizontal bands. Then, for such $j$, 
$$|F^{n+j}(M)-F^{n+j}(M')|>d,$$
which contradicts the assumption. 

Therefore, $P=M$ and $F^{n}(M')$ belongs to $W^{s}_{loc}(F^{n}(M))$. 

\bigskip
The proof of the other property is similar.  
\end{proof}

Expansion in the unstable direction by iteration of $F$ or in the stable direction by iterations of $F^{-1}$ yields the second item of the main theorem in Section~\ref{statements} 

\begin{corollary}
\label{coro-expan-shift}
The map $F$ is expansive with expansivity constant $d$. It is conjugated, by a H\"older continuous homeomorphism $\Theta_{F}$, to the subshift of finite type $\Sigma_A^9$, with 9 symbols and with the matrix $A$ given in  \eqref{trans.matrix}.  
\end{corollary}
\begin{proof}
Actually, the proof of Proposition \ref{prop-prodstruc} shows a better result:
$$M'\in W^{u}_{loc}(M)\iff \forall n\ge 1\  |F^{-n}(M)-F^{-n}(M')|<d.$$
The same result holds for $W^{s}_{loc}$ and forward iterates. This shows that $d$ is an expansivity constant. 

Let us now show that $\Theta_{F}$ is continuous. 
Let us first prove continuity at $\xi$. Let $\ul\xi$ be such that $\Theta_{F}(\ul\xi)=\xi$. Consider $\ul m$ coinciding with $\ul\xi$ from $-p$ to $q$. Set $M:=\Theta_{F}(\ul m)$.

The point $\ul m':=[\ul \xi,\ul m]$ corresponds to a point with the same past than $\xi$ and the same future than $M$. It is thus on the intersection of the stable leaf of $M$ and the unstable leaf of $\xi$. Assuming $p$ and $q$ are large, this corresponds to a point having local coordinates of the form $(0,y')$. Its image by $F$ is 
$$(-\eps_{1}y',-c{y'}^{3}).$$

The future itinerary coincides with the one of $F(\xi)$ for $q-1$ iterates, and we recall that  expansion is always greater than $\rho$ in the vertical direction (remember that $\xi$ never returns to $\mr$)
we get
$$c|y'|^{3}\rho^{q-2}\le d.$$
The larger $q$ is, the smaller $|y'|$ is. 

If we set $M=(x,y)$ (with local coordinates), $F^{-1}(M)$ stays into $R_{7}$ for $p-1$ iterates of $F^{-1}$, which yields
$$x\lambda^{p-1}\lesssim d.$$ The larger $p$ is the smaller $x$ is. This also means that $F(M)$ is below the image of the vertical line $X=x$ by $F$, which is the  cubic of equation 
$$(-\eps_{1}Y, -bx-cY(Y^{2}+x)).$$
On the other hand $F(M)$ is on the stable leaf of the point $(-c\eps_{1}y',-c{y'}^{3})$. The closer to $F(\xi)$ a point is, the thinner the stable cone is (and centered on the horizontal line), which means that the closer the stable manifold is from a horizontal straight line. 

\begin{figure}[htbp]
\begin{center}
\includegraphics[scale=0.5]{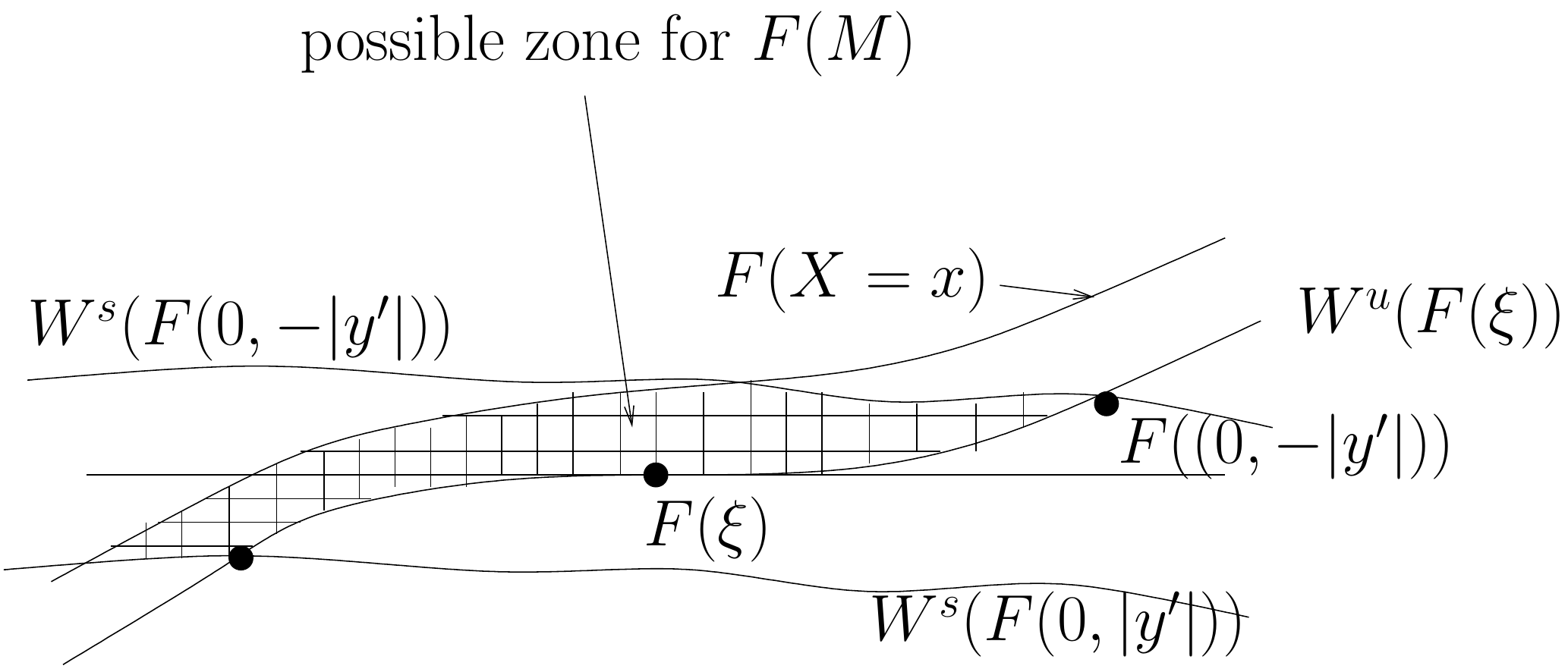}
\caption{computation of possible zone for $F(M)$}
\label{fig-intersect}
\end{center}
\end{figure}

This means that if $p$ and $q$ increase, the region (see Figure \ref{fig-intersect}) which contains $F(M)$ get smaller and concentrated around $F(\xi)$, and this proves continuity of $\theta_{F}$ at $\ul \xi$.

Continuity for points of the form $\s^{n}(\ul\xi)$ can be deduced from continuity at $\ul\xi$.

Now, consider  $M=\Theta_{F}(\ul m)$ and $M'=\Theta_{F}(\ul m')$. Assume $\ul m$ and $\ul m'$ coincide for $-p,\ldots,q$. Assume $M\notin \CO(\xi)$. We copy the arguments for the proof of the continuity at $\xi$. 

Set $M'':[M,M']$ and $M''':=[M',M]$. Roughly speaking, 
$$|M-M''|<d\kappa(M)\rho^{-\frac{q}5}\text{ and }|M-M'''|<d\kappa(M)\lambda^{-\frac{p}5}.$$
And in that case, the local unstable leaves are ``more'' vertical than for $F(\xi)$ and the local stable leaves more horizontal. 
This shows continuity for every $\ul m$ where $M=\Theta_{F}(\ul m)\notin \CO(\xi)$. 

\end{proof}

%
%
%
%


\section{Final remarks}\label{sec-border}

It remains to prove the item $1$ of the main theorem, and the corollary stated in Section~\ref{statements}.

\subsection{Heteroclinic tangency}

In our construction, the backward itinerary of $\xi$ is fixed, since $\xi$ belongs to the unstable leaf of $P_{b}:=0$. 

For the forward  itinerary we recall that our unique assumptions were that the point never comes back into $R_{7}\cup R_{4}$ and visits infinitely many times each one of the two regions $R_{1}\cup R_{2}\cup R_{3}$ and $R_{5}\cup R_{6}$. This can be realized by a periodic or non-periodic point. Chose a point $P_{f}$ with such a forward orbit, and take $F(\xi)$ in $W^{s}_{loc}(P_{f})$.

Note  that second point $P_{f}$ can be chosen in a the set of points whose trajectories satisfy the assumptions above, and this set is non-countable.


\subsection{Boundary of uniformly hyperbolic diffeomorphisms}
To finish the proof of the Theorem it remains to show that the map $F$ is on the boundary of uniformly hyperbolic diffeomorphisms.

We consider a very small parameter $\theta>0$ and then set 
$$F_{\theta}(\xi+(x,y))=F(\xi)+(-\eps_{1} y,bx-cy(y^2+x+\theta)).$$

To get an idea why $F_{\theta}$ is hyperbolic, we just point out that the parameter $\theta$ 
straightens out the cubics in the post-critical zone, thus improves the estimates. Now, we effectively prove it.

Close to the image of $\xi$, the new unstable cone field is defined by 
$$\mathsf{C}^{u}(F_{\theta}(x,y))=\{\mathbf{v }=(v_{x},v_{y}),\ |v_{y}|\ge
A(3y^2+x+\theta)|v_{x}|\}.$$

We recall that the main point to get hyperbolicity for $F$ was to prove that Inequality \eqref{equ4-cone-instable} held, \ie for $u\ge 1$
$$
\left|\frac{\lambda^{k_{0}}(b-c\beta)}{\s^{k_{0}^{(1)}}\rho^{k_{0}^{(2)}}Au(3y^2+x)\eps_{1}}-\frac{c}{
\eps_{1}}(3\beta^2+\alpha)\right|\ge A(3\beta^2+\alpha).
$$
The way to prove this was to show that the first term in the left hand side was smaller than the half of the second term, and that half of the second term (in the left hand size) was larger than the term in the right hand side. 

Introducing the parameter $\theta>0$ we have to prove 

\begin{equation}
\label{equ-cone-borderunif}
\left|\frac{\lambda^{k_{0}}(b-c\beta)}{\s^{k_{0}^{(1)}}\rho^{k_{0}^{(2)}}Au(3y^2+x+\theta)\eps_{1}}-\frac{c}{
\eps_{1}}(3\beta^2+\alpha+\theta)\right|\ge A(3\beta^2+\alpha+\theta).
\end{equation}

Note that  for $\theta>0$, if 
$$0\le \frac{\lambda^{k_{0}}(b-c\beta)}{\s^{k_{0}^{(1)}}\rho^{k_{0}^{(2)}}Au(3y^2+x)\eps_{1}}\le
\frac{c}{2\eps_{1}}(3\beta^2+\alpha)$$
holds, then 
$$0\le \frac{\lambda^{k_{0}}(b-c\beta)}{\s^{k_{0}^{(1)}}\rho^{k_{0}^{(2)}}Au(3y^2+x+\theta)\eps_{1}}\le
\frac{c}{2\eps_{1}}(3\beta^2+\alpha)\le \frac{c}{2\eps_{1}}(3\beta^2+\alpha+\theta)$$
also holds.
Consequently \eqref{equ-cone-borderunif} holds.

It is left to the reader to check that any other required condition holds for $\theta$ sufficiently small, as $|y|\le \be_{max}$ very small and we always considered strict inequalities to allow some small perturbation. 

We also point out that exchanging $-cyx$ by $-c(x+\theta)y$ does not change the result  of the uniqueness of the real root when we used Cardan's method. In fact, it improves the result. 

It now remains to show that $F_{\theta}$ is uniformly hyperbolic. This follows from the next lemma:
\begin{lemma}
\label{lem-unifbordereu}
There exists $N=N(\theta)$ such that for every $M$ in $\Lambda$ there exists $0\le n\le N$ such that for every $\mbf{v}\in\mathsf{C}^{u}(M)$, $|DF_{\theta}^{n}(M).\mbf{v}|\ge 2|\mbf{v}|$. 
\end{lemma}
\begin{proof}
Note that any vector in some unstable cone is more vertical than $(1,A\theta)$. In the following, we consider $N'$ such that $A\theta\rho^{N'}>2$. 

Now consider any point $M$ and any vector $\mbf{v}\in\mathsf{C}^{u}(M)$. Assume $M$ is very close to $F(\xi)$. We call $n_{esc}(M)$ the maximal integer such that $|F^{j+1}(M)-F^{j+1}(\xi)|<d$  holds for every $j\le n_{esc}(M)$. 

If $n_{esc}(M)\ge N'$, then  $|DF_{\theta}^{N'}(M).\mbf{v}|\ge 2|v|$. 

If  $n_{esc}(M)<N'$, after $n_{esc}(M)$ iterates Inequality\eqref{equ2.2-expansion 0-m} shows that $\mbf{v}$ has been expanded by $DF_{\theta}^{n_{esc}(M)}(M)$ by a factor greater than $\disp\frac{Ad\eps_{1}}{3c\be_{max}}$. This last term is supposed much larger than $1$ (see condition \eqref{equ-grandApour expansion}) and can thus be assumed larger than 2. 

If $M$ is not in $F(\mr)$ and needs at least than $N'$ iterates to reach $\mr$ then we are done. 
If $M$ is not in $F(\mr)$ and needs less that $N'$ iterates to reach $\mr$, we have to consider several cases.

If $M$ is in the post-critical part of some critical tube, that is $M=F^{k}(M')$, with $k\le n_{+}(M')$ (see Def. \ref{def-crit-tub}) then 2 situations occur :
Either  the vectors at  border of the unstable cones are contracted by $DF(M)$, but when the point exits the critical time (which happens before $N'$ iterates because $F^{N'}(M)$ belongs to $\mr$ and $\xi$ never returns to $\mr$), the vectors have been expanded (case $n_{esc}(M)<N'$). 
Or the vectors at the border of the unstable cones are expanded and then, the worst case is than we need $N'$ more iterates after reaching $\mr$ to be sure we get an expansion greater than 2. 
Therefore, $N(\theta)=2N'$  satisfies the condition. 
\end{proof}

We claim that the factor $\theta$ implies that $DF_{\theta}(\mathsf{C}^{u})$ is  uniformly included in the interior of $\mathsf{C^{u}}$ as illustrated on Figure \ref{fig-cone-unif}. The security angle is uniformly proportional to $\theta$. Therefore, the angle between $E^{u}$ and $E^{s}$ is uniformly bounded away from zero. The assumption $|\det(DF_{\theta})|<1$ yields uniform contraction in the stable direction, and $F_{\theta}$ is uniformly hyperbolic. Moreover, the stable and the unstable vector fields are Lipschitz when restricted to stable or unstable local manifolds. 

\begin{figure}[htbp]
\begin{center}
\includegraphics[scale=.4]{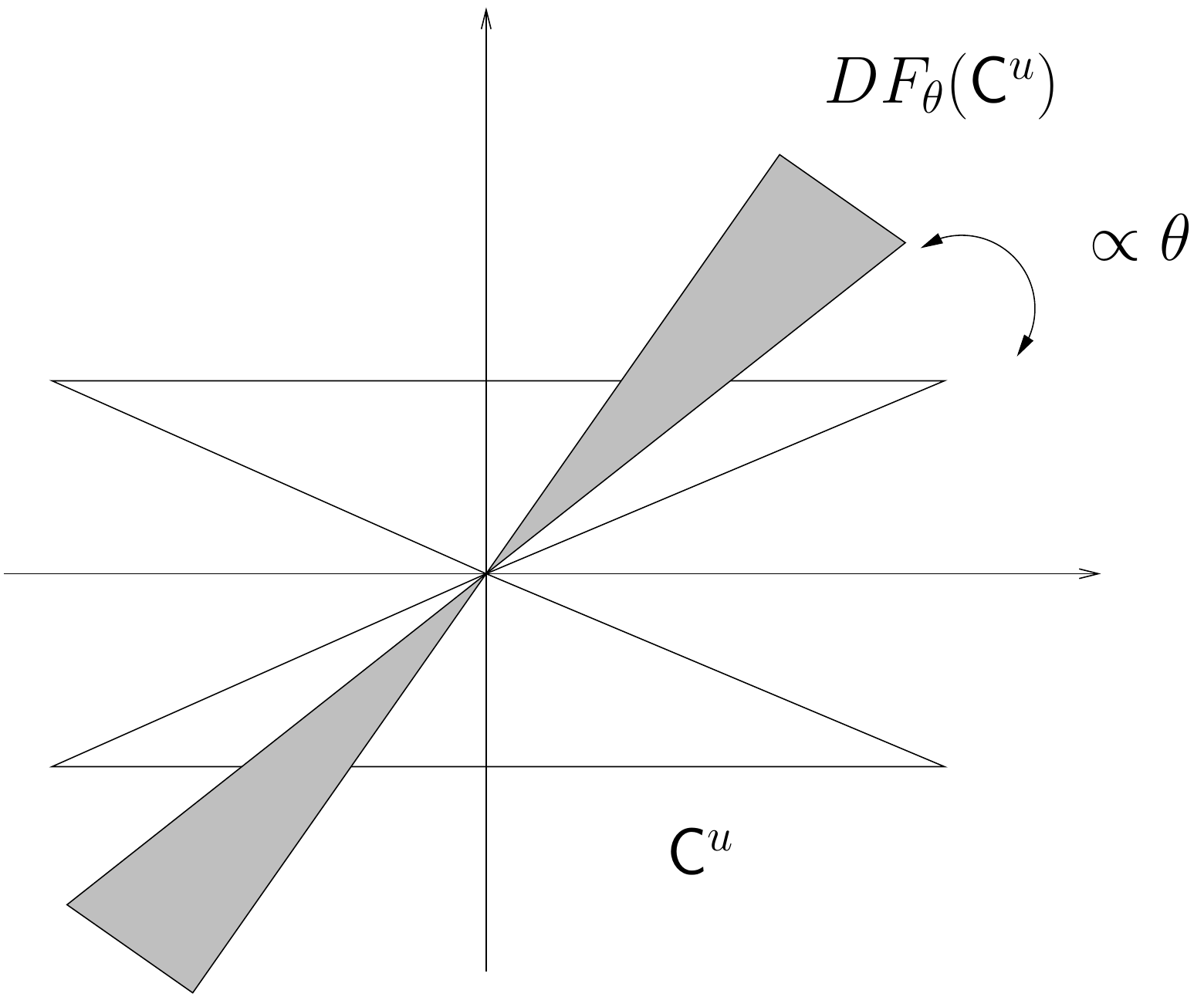}
\caption{Image of unstable cone strictly inside the unstable cone}
\label{fig-cone-unif}
\end{center}
\end{figure}

This finishes the proof of the Theorem. 

\subsection{Equilibrium states}

  We say 
that a $F$-invariant probability measure $\mu$ is an {\em equilibrium state} for $F$ w.r.t. a 
potential $\phi: \Lambda \to \mathbb{R}$ if it satisfies
          $$h_{\mu} (F) + \int \phi \, d\mu  = \sup_\eta \left\{ h_{\eta} (F) + \int \phi \, d\eta 
\right\},  $$
where the supremum is taken over all $F$-invariant probability measures. 

The H\"older continuous conjugacy $\Theta_{F}:\Sigma_A^9\to\Lambda$ allows us to define a H\"older continuous potential $\tilde\phi:\Sigma_A^9\to \R$ by 
$\tilde\phi(\ul x)=\phi(\Theta_{F}(\ul x))$.

The existence and uniqueness of equilibrium states for subshifts of finite type with respect to H\"older continuous potentials is a classical result of Ergodic Theory (see \cite{Ruelle78}). The conjugacy gives a correspondence between the invariant measures of the two systems. If $\mu_{\tilde\phi}$ is the unique equilibrium measure for the subshift, then $\Theta_{F}^*(\mu_{\tilde\phi})$ is the unique equilibrium state for $F$.

\subsection{Acknowledgments}
This work was partially supported by CNRS-CNPq cooperation, FP7-Irses 230844 DynEurBraz, FP7-Irses 318999 Breuds. 
Authors also want to deeply thank J. Rivera-Letelier for his help to find the right explicit form of the cubic tangency.

\bibliographystyle{plain}
\bibliography{mabiblio}

\begin{thebibliography}{10}

\bibitem{alves}
J.~F. Alves.
\newblock S{RB} measures for non-hyperbolic systems with multidimensional
  expansion.
\newblock {\em Ann. Sci. \'Ecole Norm. Sup. (4)}, 33(1):1--32, 2000.

\bibitem{alves-bonatti-viana}
J.F. {Alves}, C.~{Bonatti}, and M.~{Viana}.
\newblock Srb measures for partially hyperbolic systems whose central direction
  is mostly expanding.
\newblock {\em Inventiones Math.}, 140:351--298, 2000.

\bibitem{young-benedicks}
M.~{Benedicks} and L.-S. {Young}.
\newblock Sinai-bowen-ruelle measures for certain h\'enon maps.
\newblock {\em Invent. Math.}, 112(3):541--576, 1993.

\bibitem{generic}
C.~Bonatti, L.~J. D{\'{\i}}az, and E.~R. Pujals.
\newblock A {$C^1$}-generic dichotomy for diffeomorphisms: weak forms of
  hyperbolicity or infinitely many sinks or sources.
\newblock {\em Ann. of Math. (2)}, 158(2):355--418, 2003.

\bibitem{bonatti-diaz-viana}
C.~Bonatti, L.~J. D{\'{\i}}az, and M.~Viana.
\newblock {\em Dynamics beyond uniform hyperbolicity}, volume 102 of {\em
  Encyclopaedia of Mathematical Sciences}.
\newblock Springer-Verlag, Berlin, 2005.
\newblock A global geometric and probabilistic perspective, Mathematical
  Physics, III.

\bibitem{Bonatti-Diaz-Vuillemin}
C.~Bonatti, L.~J. D{\'{\i}}az, and F.~Vuillemin.
\newblock Cubic tangencies and hyperbolic diffeomorphisms.
\newblock {\em Bol. Soc. Brasil. Mat. (N.S.)}, 29(1):99--144, 1998.

\bibitem{bruin-todd}
H.~Bruin and M.~Todd.
\newblock Markov extensions and lifting measures for complex polynomials.
\newblock {\em Ergodic Theory Dynam. Systems}, 27(3):743--768, 2007.

\bibitem{buzzi-thermo}
J.~Buzzi.
\newblock Thermodynamical formalism for piecewise invertible maps: absolutely
  continuous invariant measures as equilibrium states.
\newblock In {\em Smooth ergodic theory and its applications ({S}eattle, {WA},
  1999)}, volume~69 of {\em Proc. Sympos. Pure Math.}, pages 749--783. Amer.
  Math. Soc., Providence, RI, 2001.

\bibitem{varandas}
A.~Castro and P.~Varandas.
\newblock Equilibrium states for non-uniformly expanding maps: decay of
  correlations and strong stability.
\newblock {\em Ann. Inst. H. Poincar\'e Anal. Non Lin\'eaire}, 30(2):225--249,
  2013.

\bibitem{Enrich}
H.~Enrich.
\newblock A heteroclinic bifurcation of {A}nosov diffeomorphisms.
\newblock {\em Ergodic Theory Dynam. Systems}, 18(3):567--608, 1998.

\bibitem{horita-muniz-sabini}
V.~Horita, N.~Muniz, and P.~Rog{\'e}rio Sabini.
\newblock Non-periodic bifurcations for surface diffeomorphisms.
\newblock {\em Trans. Amer. Math. Soc}, 367:8279--8300, 2015.

\bibitem{hu}
H.~{Hu}.
\newblock Conditions for the existence of sbr measures for {`Almost Anosov
  Diffeomorphisms'}.
\newblock {\em Transaction of AMS}, 1999.

\bibitem{LOR}
R.~Leplaideur, K.~Oliveira, and I.~Rios.
\newblock Equilibrium states for partially hyperbolic horseshoes.
\newblock {\em Ergodic Theory Dynam. Systems}, 31(1):179--195, 2011.

\bibitem{isareno1}
R.~Leplaideur and I.~Rios.
\newblock Invariant manifolds and equilibrium states for non-uniformly
  hyperbolic horseshoes.
\newblock {\em Nonlinearity}, 19(11):2667--2694, 2006.

\bibitem{isareno2}
R.~Leplaideur and I.~Rios.
\newblock On {$t$}-conformal measures and {H}ausdorff dimension for a family of
  non-uniformly hyperbolic horseshoes.
\newblock {\em Ergodic Theory Dynam. Systems}, 29(6):1917--1950, 2009.

\bibitem{oliveira}
K.~Oliveira and M.~Viana.
\newblock Thermodynamical formalism for robust classes of potentials and
  non-uniformly hyperbolic maps.
\newblock {\em Ergodic Theory Dynam. Systems}, 28(2):501--533, 2008.

\bibitem{palis-takens}
J.~Palis and F.~Takens.
\newblock {\em Hyperbolicity and sensitive chaotic dynamics at homoclinic
  bifurcations}, volume~35 of {\em Cambridge Studies in Advanced Mathematics}.
\newblock Cambridge University Press, Cambridge, 1993.
\newblock Fractal dimensions and infinitely many attractors.

\bibitem{yoccoz-palis}
J.~Palis and JC~Yoccoz.
\newblock Non-uniformly hyperbolic horseshoes arising from bifurcations of
  {P}oincar\'e heteroclinic cycles.
\newblock {\em Publ. Math. Inst. Hautes \'Etudes Sci.}, (110):1--217, 2009.

\bibitem{pinheiro}
V.~Pinheiro.
\newblock Expanding measures.
\newblock {\em Ann. Inst. H. Poincar\'e Anal. Non Lin\'eaire}, 28(6):889--939,
  2011.

\bibitem{rios1}
I.~Rios.
\newblock Unfolding homoclinic tangencies inside horseshoes: hyperbolicity,
  fractal dimensions and persistent tangencies.
\newblock {\em Nonlinearity}, 14(3):431--462, 2001.

\bibitem{Ruelle78}
D.~Ruelle.
\newblock {\em Thermodynamic formalism}, volume~5 of {\em Encyclopaedia of
  Mathematics and its Applications}.
\newblock Addison-Wesley Publishing Company, 1978.

\bibitem{senti-takahasi}
S.~Senti and H.~Takahasi.
\newblock Equilibrium measures for the {H}\'enon map at the first bifurcation.
\newblock {\em Nonlinearity}, 26(6):1719--1741, 2013.

\end{thebibliography}
\listoffigures

%

\end{document}